\RequirePackage[l2tabu,orthodox,abort]{nag}

\documentclass{birkjour}[draft]

\usepackage{fixltx2e}

\usepackage{amssymb,amsfonts,amsmath,
amsrefs,enumerate,color,url,
mathbbol,dsfont}
\usepackage{mathrsfs}
\usepackage{tikz}
\usepackage{showkeys}

\usepackage[T1]{fontenc}
\usepackage[cp1250]{inputenc}

\usepackage[normalem]{ulem}

\usepackage[strict=true]{csquotes}

\usepackage{graphicx,tikz}

 \newtheorem{thm}{Theorem}[section]
 \newtheorem{cor}[thm]{Corollary}
 \newtheorem{lem}[thm]{Lemma}
 \newtheorem{prop}[thm]{Proposition}
 \theoremstyle{definition}
 \newtheorem{defn}[thm]{Definition}
 \theoremstyle{remark}
 \newtheorem{rem}{Remark}
 
 \numberwithin{equation}{section}
 \newcommand{\aab}{A_{\alpha,\beta}}
\newcommand{\anab}{A_{n\alpha,n\beta}}
\newcommand{\askew}{A^{\textnormal {{skew}}}_{\alpha,\beta}}
\newcommand{\aweks}{A^{\textnormal {{weks}}}_{\alpha,\beta}}
\newcommand{\xreg}{\mathsf X_{\textnormal {reg}}}

\newcommand{\enab}{E_{n\alpha,n\beta}}
 \DeclareMathOperator{\cosab}{Cos_{\alpha,\beta}}
 \DeclareMathOperator{\cosabt}{Cos_{\alpha,\beta}^\perp}
 \DeclareMathOperator{\cosabn}{Cos_{n\alpha,n\beta}}
 \DeclareMathOperator{\cosabnt}{Cos_{n\alpha,n\beta}^\perp}
\DeclareMathOperator{\cskew}{Cos_{\alpha,\beta}^{\textnormal{skew}}}
\DeclareMathOperator{\cskewt}{Cos_{\beta,\alpha}^{\textnormal{skew}}}

\DeclareMathOperator{\cosik}{Cos}
\usepackage{mathtools}
\usepackage{enumitem}
\renewcommand{\theenumi}{\textup{(\alph{enumi})}}
\renewcommand{\labelenumi}{\theenumi}

\usepackage{hyperref}
\begin{document}

\title[Complementary transmission conditions]{On pairs of complementary transmission conditions and on approximation of skew Brownian motion by snapping-out Brownian motions}

%----------Author 1
\author[A. Bobrowski]{Adam Bobrowski}

\address{
Lublin University of Technology\\
Nadbystrzycka 38A\\
20-618 Lublin, Poland}
\email{a.bobrowski@pollub.pl}

\author[E. Ratajczyk]{El\.zbieta Ratajczyk}
\address{
Lublin University of Technology\\
Nadbystrzycka 38A\\
20-618 Lublin, Poland}

\email{e.ratajczyk@pollub.pl}

\newcommand{\cxi}{(\xi_i)_{i\in \N} }
\newcommand{\lam}{\lambda}
\newcommand{\eps}{\varepsilon}
\newcommand{\ud}{\, \mathrm{d}}
\newcommand{\mud}{\mathrm{d}}
\newcommand{\pr}{\mathbb{P}}
\newcommand{\f}{\mathcal{F}}
\newcommand{\s}{\mathcal{S}}
\newcommand{\h}{\mathcal{H}}
\newcommand{\ai}{\mathcal{I}}
\newcommand{\R}{\mathbb{R}}
\newcommand{\C}{\mathbb{C}}
\newcommand{\Z}{\mathbb{Z}}
\newcommand{\N}{\mathbb{N}}
\newcommand{\Y}{\mathbb{Y}}
\newcommand{\e}{\mathrm {e}}
\newcommand{\tif}{\widetilde {f}}
\newcommand{\Id}{{\mathrm{Id}}}
\newcommand{\cic}{C_{\mathrm{mp}}}
\newcommand{\pol}{{\textstyle \frac 12}}
\newcommand{\es}{\textnormal T}

\newcommand{\cee}{\mathfrak  C[-\infty,\infty]}
\newcommand{\cod}{\mathfrak C_{\mathrm{odd}}[-\infty,\infty]}
\newcommand{\cev}{\mathfrak C_{\mathrm{even}}[-\infty,\infty]}
\newcommand{\cevr}{C_{\mathrm{even}}(\mathbb{R})}
\newcommand{\codr}{C_{\mathrm{odd}}(\mathbb{R})}
\newcommand{\cez}{C_0(0,1]}
\newcommand{\fod}{f_{\mathrm{odd}}} 
\newcommand{\fev}{f_{\mathrm{even}}} 
\newcommand{\sem}[1]{\mbox{$\left (\e^{t{#1}}\right )_{t \ge 0}$}}
\newcommand{\semi}[1]{\mbox{$\left ({#1}\right )_{t > 0}$}}
\newcommand{\semt}[2]{\mbox{$\left (\e^{t{#1}} \otimes_\varepsilon \e^{t{#2}} \right )_{t \ge 0}$}}
\newcommand{\tr}{\textcolor{red}}
\newcommand{\wt}{\widetilde}

\newcommand{\tcm}{\textcolor{magenta}}
\newcommand{\ecm}{\textcolor{olive}}
\newcommand{\tcb}{\textcolor{blue}}
\newcommand{\dx}{\ \textrm {d} x}
\newcommand{\dy}{\ \textrm {d} y}
\newcommand{\dz}{\ \textrm {d} z}
\newcommand{\di}{\textrm{d}}
\newcommand{\tcg}{\textcolor{green}}
\newcommand{\lc}{\mathfrak L_c}
\newcommand{\ls}{\mathfrak L_s}
\newcommand{\grat}{\lim_{t\to \infty}}
\newcommand{\grar}{\lim_{r\to 1-}}
\newcommand{\graR}{\lim_{R\to 1+}}
\newcommand{\grak}{\lim_{\kappa \to \infty}}
\newcommand{\gra}{\lim_{x\to \infty}}
\newcommand{\grae}{\lim_{\eps \to 0}}
\newcommand{\gran}{\lim_{n\to \infty}}
\newcommand{\rez}[1]{\left (\lam - #1\right)^{-1}}
\newcommand{\papa}{\hfill $\square$}
\newcommand{\papap}{\end{proof}}
\newcommand {\x}{\cerf}
\newcommand{\aex}{A_{\mathrm ex}}
\newcommand{\jcg}[1]{\left ( #1 \right )_{n\ge 1} }
\newcommand{\injtp}{\x \hat \otimes_{\varepsilon} \y}
\newcommand{\pin}{\|_{\varepsilon}}
\newcommand{\mc}{\mathcal}
\newcommand{\inter}{\left [0, 1\right ]}
\newcommand{\ha}{\mathfrak {H}}
\newcommand{\dom}[1]{D(#1)}
\newcommand{\mquad}[1]{\quad\text{#1}\quad}
\newcommand{\lil}{\lim_{\lam \to \infty}}
\newcommand{\lilz}{\lim_{\lam \to 0}}

\newcommand{\ce}{\mathcal C}
\newcommand{\cerl}{\mathfrak C[-\infty,0]}

\newcommand{\cerp}{\mathfrak C[0,\infty]}
\newcommand{\cer}{\mathfrak  C[-\infty,\infty]}
\newcommand{\cerf}{\mathfrak  C(\R_\sharp)}
\newcommand {\xo}{\mathfrak  C_{\textnormal{ov}}(\R_\sharp)}
\newcommand{\cerr}{\mathfrak C_R^\alpha}
\newcommand{\cef}{\mathfrak C_F^\alpha}
\newcommand{\comega}{C_\omega[-\infty,\infty]}
\newcommand{\fo}{f_{\textrm{o}}}
\newcommand{\fe}{f_{\textrm{e}}}
\newcommand{\wh}{\widehat}

%moje
\newcommand{\sq}{\gamma}
\newcommand{\su}{{\alpha+\beta}}
\newcommand{\al}{\alpha}
\newcommand{\be}{\beta}

\newcommand{\yy}{\mathcal{C}_{\al,\be}^{\textnormal{skew}}}
\newcommand{\zz}{\mathcal{D}_{\al,\be}^{\textnormal{weks}}}
\newcommand{\pp}{P_{\al,\be}^{\textnormal{skew}}}
\renewcommand{\qq}{Q_{\al,\be}^{\textnormal{weks}}}
\newcommand{\y}{\mathcal{C}_{\al,\be}}
\newcommand{\z}{\mathcal{D}_{\al,\be}}
\renewcommand{\q}{Q_{\al,\be}}
\newcommand{\p}{P_{\al,\be}}

\newcommand{\fl}{f_{\ell}}
\newcommand{\fr}{f_{\textnormal r}}
\newcommand{\fle}{f_1}
\newcommand{\fre}{f_2}
\newcommand{\wfl}{\widetilde{\fl}}
\newcommand{\wfr}{\widetilde{\fr}}
\newcommand{\gl}{g_1}
\newcommand{\gr}{g_2}
\newcommand{\gle}{g_1}
\newcommand{\gre}{g_2}
\newcommand{\hl}{h_{\textnormal l}}
\newcommand{\hr}{h_{\textnormal r}}
\newcommand{\apb}{{\alpha+\beta}}
\newcommand{\Ff}{\mathfrak{F}}
\newcommand{\Ffe}{k}
\newcommand{\F}{\mathcal{F}}

\newcommand{\rla}{R_\lam}

\newcommand{\fh}{h} % oznaczenie na funkcje pomocnicze funkcje h w rozdziale 3.1
\newcommand{\ph}{\phi} % oznaczenie na funkcje (dawniej oznaczanie przez h) zdefiniowane w lemacie 3.2

\newcommand{\ced}{C_{\textnormal{D}}}

\makeatletter
\newcommand{\normt}{\@ifstar\@normts\@normt}
\newcommand{\@normts}[1]{%
  \left|\mkern-1.5mu\left|\mkern-1.5mu\left|
   #1
  \right|\mkern-1.5mu\right|\mkern-1.5mu\right|
}
\newcommand{\@normt}[2][]{%
  \mathopen{#1|\mkern-1.5mu#1|\mkern-1.5mu#1|}
  #2
  \mathclose{#1|\mkern-1.5mu#1|\mkern-1.5mu#1|}
}
\makeatother

\thanks{Version of \today}
%----------Au

\begin{abstract}Following our previous work on `perpendicular' boundary conditions, we show that transmission conditions  
\begin{align*}
f'(0-)=\al(f(0+)-f(0-)), \qquad 
f'(0+)=\be(f(0+)-f(0-)),
\end{align*}
describing so-called snapping out Brownian motions on the real line, are in a sense complementary to  
the transmission conditions 
\begin{align*}f(0-)=-f(0+), \qquad 
f''(0+) =\al f'(0-)+\be f'(0+). 
 \end{align*}
As an application of the analysis leading to this result, we also provide a deeper semigroup-theoretic insight into the theorem saying that as the coefficients $\alpha$ and $\beta$ tend to infinity but their ratio remains constant, the snapping-out Brownian motions converge to a skew Brownian motion. In particular, the transmission condition 
\[ \alpha f'(0+) = \beta f'(0-), \]
that characterizes the skew Brownian motion turns out to be complementary to 
\[ f(0-) = - f(0+), \beta f'(0+)=- \alpha f'(0-). \] 
\end{abstract}

\maketitle

\section{Introduction}

\subsection{Boundary conditions and invariant subspaces}
There is an intimate connection between boundary conditions for one-di\-men\-sion\-al Laplace operator $f\mapsto f''$ and invariant subspaces for the \emph{basic cosine family} $ \{ C(t),\, t \in \R\}$ defined by  
\begin{equation}\label{intro:0} C(t) f(x) = \pol [ f(x+t) + f(x-t) ], \qquad x \in \R, t \in \R.\end{equation}

To explain this, let $\cer$ be the space of continuous functions on the real line that have finite limits at plus and minus infinity; this space is equipped, as customary, with the supremum norm. The subspaces \[ \cod\subset \cee \mquad {and} \cev\subset \cee \] of odd and even functions, respectively, are invariant under  $ \{ C(t),\, t \in \R\}$, as seen as a family of operators in $\cee$.  This has an immediate bearing on generation theorems: Since the space of even functions is isometrically isomorphic to the space $\cerp$ of continuous functions on the non-negative half line that have limits at infinity, and since the even extension of a twice continuously differentiable function $f$ on $[0,\infty)$ is twice continuously differentiable on the entire line iff $f'(0)=0$, invariance of the space of even functions allows proving that the Laplace operator in $\cerp$ with domain described by the Neumann boundary condition $f'(0)=0$ is a cosine family generator. To wit, the cosine family generated by the latter operator can be given explicitly: 
\begin{equation}\label{intro:1} \operatorname{Cos}(t) \coloneqq RC(t)E, \qquad  t \in \R \end{equation}
where $E\colon\cerp\to\cer$ maps a function to its even extension, and $R\colon\cee \to \cerp$ maps a function to its restriction. The same analysis allows linking invariance of the subspace of odd functions with the Dirichlet boundary condition $f(0)=0$ --- this method of proving generation theorem is referred to as Lord Kelvin's method of images (see e.g. \cite{kelvin,bobmug,bobgre,bobgremur}; in \cite{feller}*{pp. 340-343} or \cite{kniga}*{Section 8.1} this method is used to prove generation theorems for the related semigroups, not the cosine families, but the analysis is analogous.)

This example is just the tip of the iceberg: as established in \cite{kosinusy}, \emph{nearly all Feller--Wentzel boundary conditions are related to invariant subspaces} of $\cer$, and, consequently, the cosine families (and semigroups) involved are given by the abstract Kelvin formula like \eqref{intro:1}, except that extension operator $E$ is of different form. For instance, the classical Robin boundary condition \begin{equation}\label{intro:2} f'(0)=\al f(0) \end{equation}  (where $\alpha\ge 0$ is a constant) 
 is linked with the invariant subspace of $g\in \cer$ that satisfy 
\begin{equation}\label{intro:3} g(-x) = g(x) - 2\al \int_0^x \e^{-\al (x-y)} g(y) \ud y, \qquad  x\ge 0. \end{equation}
\begin{center}
\begin{figure} 
\includegraphics[scale=0.24]{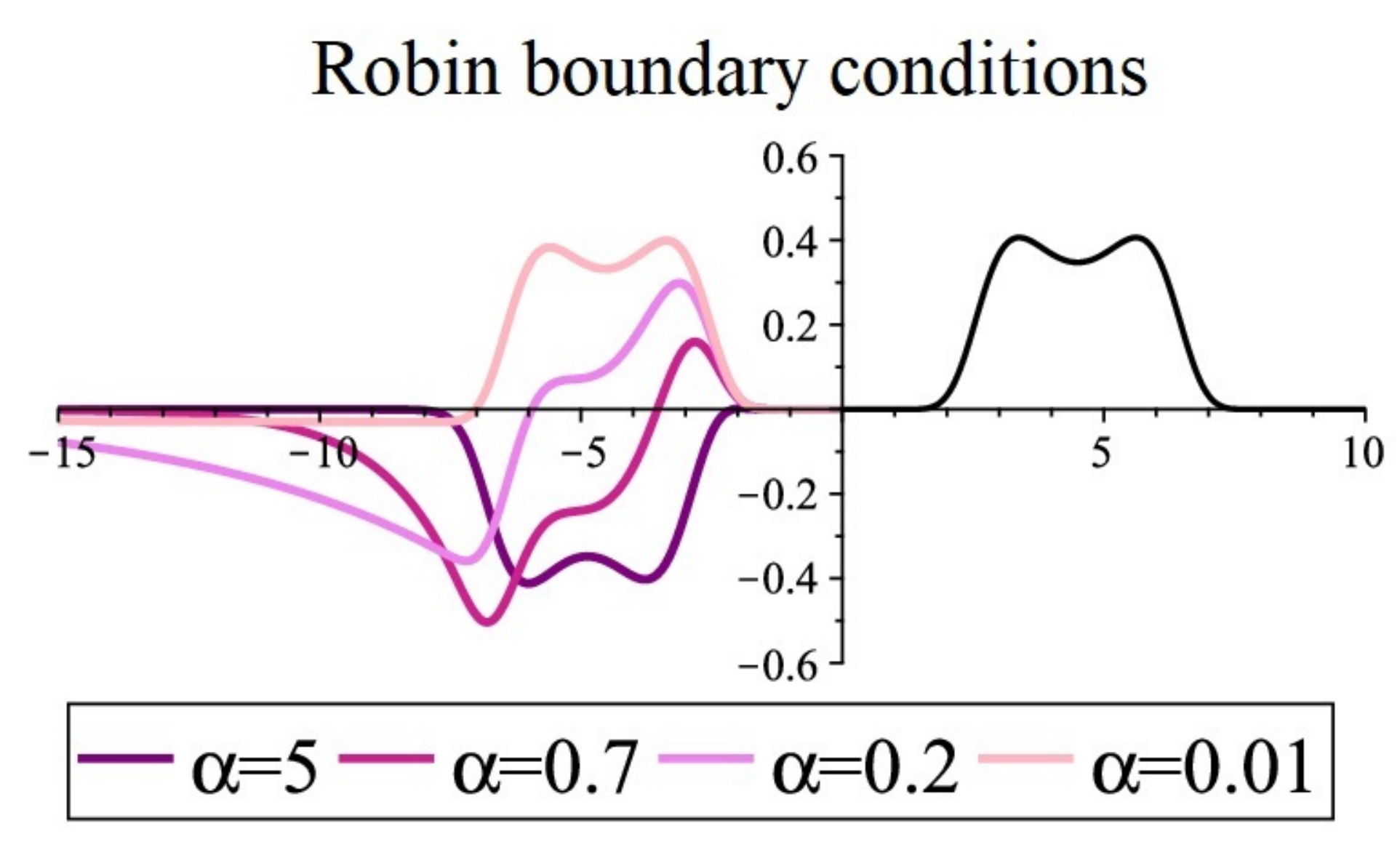}
\includegraphics[scale=0.24]{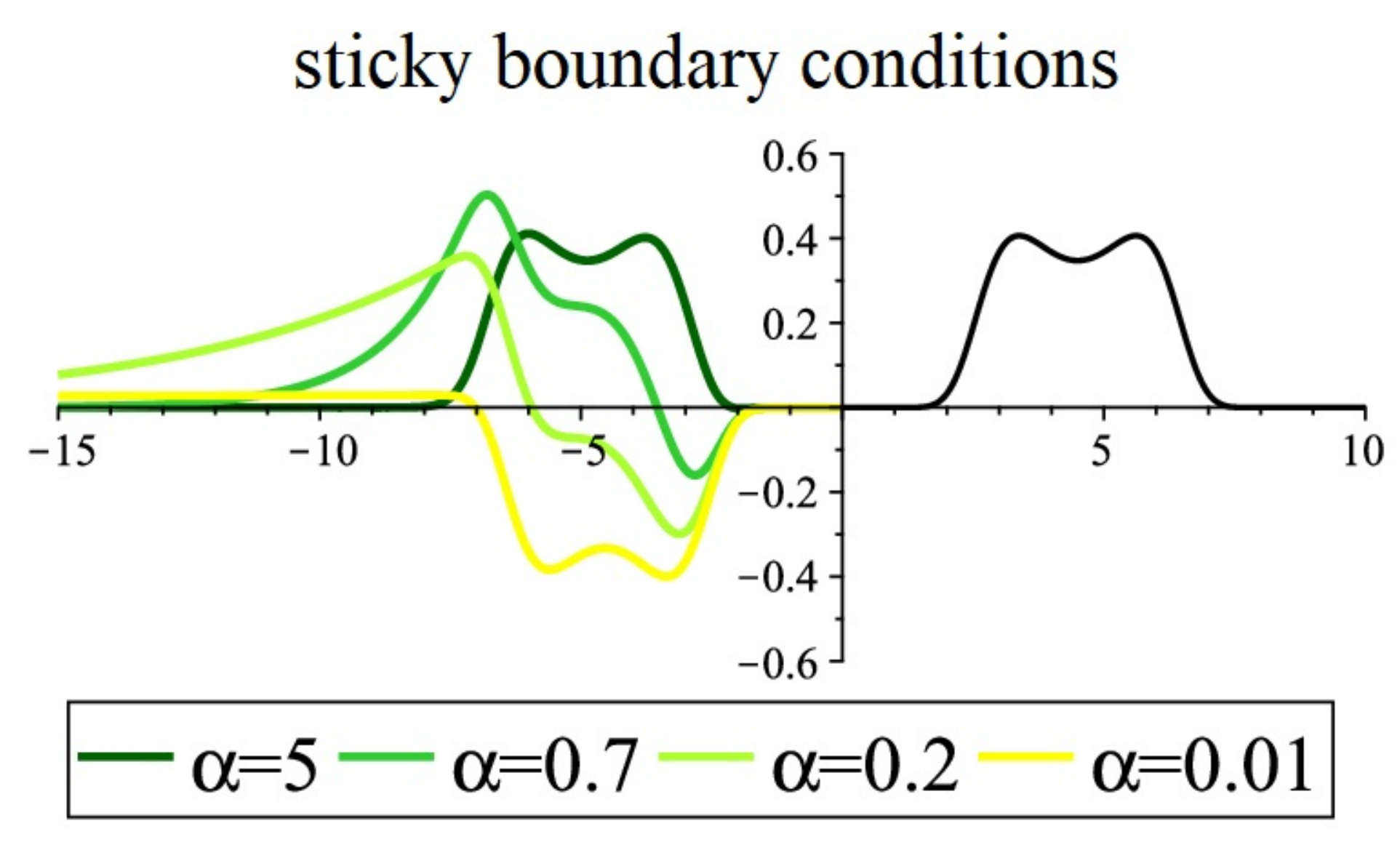}
\caption{Robin and Feller extensions of exemplary functions}
\label{rys1}
\end{figure}
\end{center}
\subsection{Complementary boundary conditions}
The formula 
\[ \cer = \cev \oplus \cod \]
says that $\cer$ can be decomposed into two subspaces that are not only invariant but also \emph{complementary}, and this leads to the conclusion that the related Neumann and Dirichlet boundary conditions are in a sense complementary as well.
In fact, these conditions are nearly `perpendicular', because projections on $\cod$ and $\cev$, mapping a function to its odd and even parts, respectively, are inherited from the space of square integrable functions where they are orthogonal projections.  

In this context, a natural question arrises whether there are any other decompositions of $\cer$ into complementary invariant subspaces, preferably subspaces related to boundary conditions in the sense described above. In other words, are there any other complementary boundary conditions, besides those of Neumann and Dirichlet? The answer given in \cite{bobrat}, is in affirmative: for $\alpha>0$,  the Robin boundary condition \eqref{intro:2} is complementary to the boundary condition (see Figure \ref{rys1})
\begin{equation} f'' (0) = \alpha f'(0),\label{intro:4} \end{equation}
related to the \emph{slowly reflecting boundary} \cite{revuz}*{p. 421} (known also as \emph{sticky boundary} \cite{liggett}*{p. 127}), a particular case of Feller boundary conditions. To repeat, this means that (a) the related subspaces, say, $\cerr$ and $\cef$ (`R' for `Robin', `F' for Feller), are 
 invariant under the basic cosine family (and thus, by the Weierstrass formula --- see e.g. \cite{abhn}*{p. 219} --- under the heat semigroup as well), (b) $\cer$ is decomposed into these subspaces as follows 
\begin{equation}\cer = \cerr \oplus \cef , \label{intro:5} \end{equation} 
and (c) the corresponding projections $P_\alpha\colon\cer \to \cerr$ and $Q_\alpha=I-P_\alpha\colon \cer \to \cef$  have a lot in common with orthogonal projections in Hilbert spaces of square integrable functions. Moreover, $P_\alpha,\alpha\ge 0$ is a continuous family, leading from the projection on the subspace of even functions to the projection on the subspace of the odd functions, whereas  $Q_\alpha,\al\ge 0$ leads in the other direction, and this via a completely different route 
(see Figure \ref{rys2}).
\begin{center}
\begin{figure} 
\includegraphics[scale=0.27]{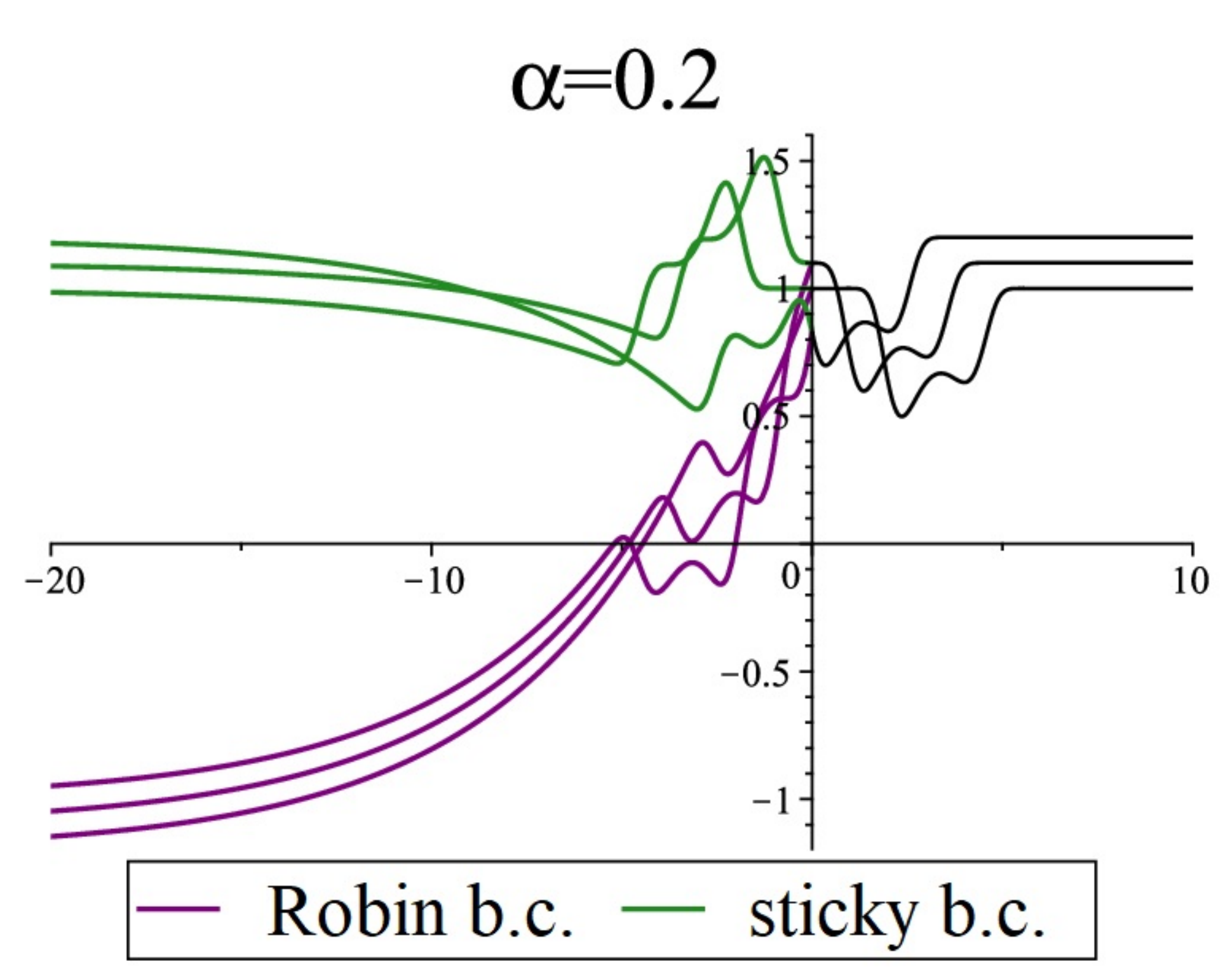}
\includegraphics[scale=0.27]{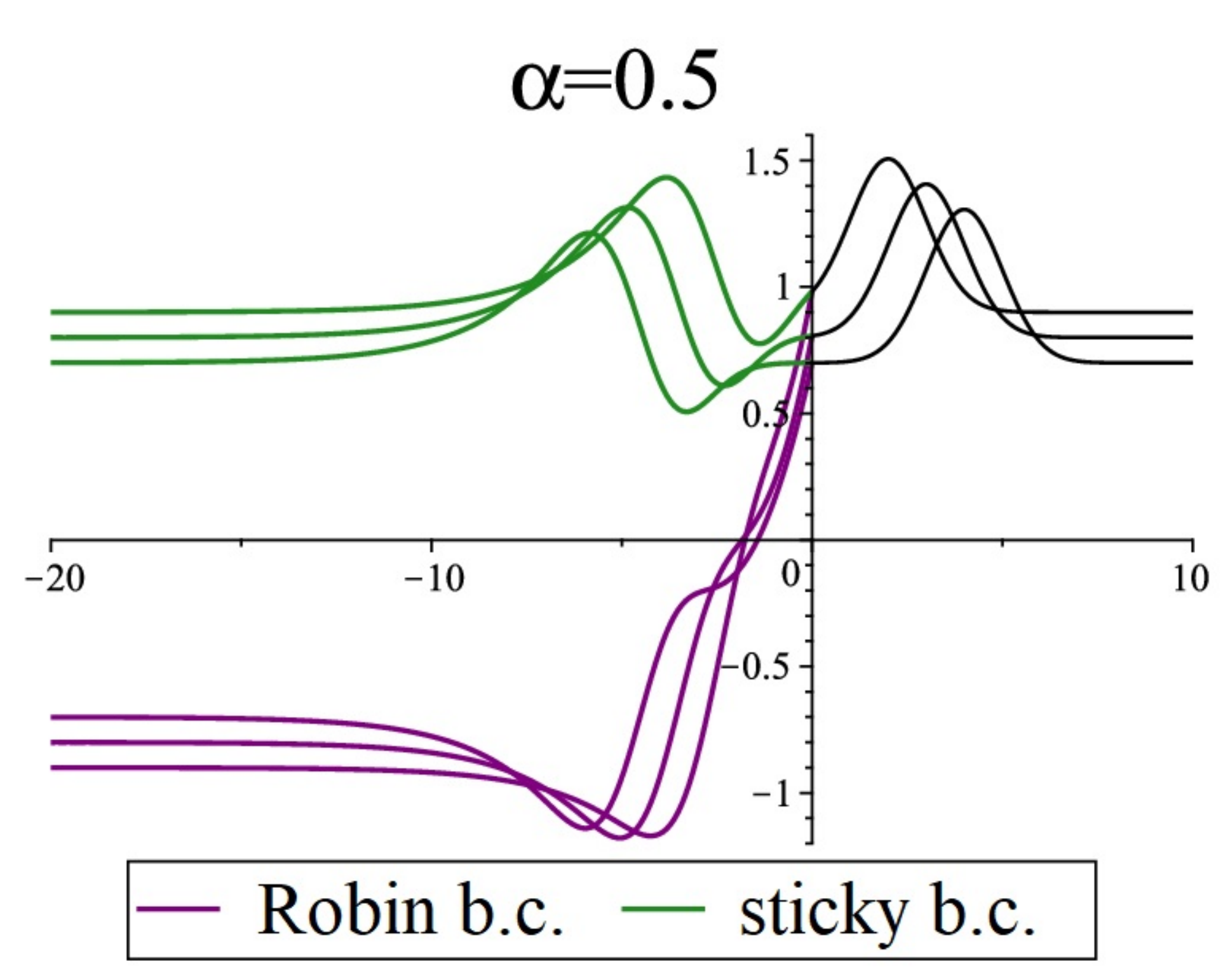}
\caption{Dependence of Robin and Feller extensions on parameter $\alpha$}
\label{rys2}
\end{figure}
\end{center}
\subsection{Invariant subspaces and transmission conditions}\label{isat}
The theory is not restricted to boundary conditions: also a number of \emph{transmission} conditions is related to invariant subspaces --- see e.g. \cite{tombatty,emergence}. For instance, as proved in \cite{tombatty}, the method of images can be used in the case of transmission conditions describing so-called \emph{snapping out Brownian motion} -- see \cites{andrews,fireman,lejayn,tanner,powles} and \cite{knigaz}*{Chapters 4 and 11}.

The snapping out Brownian motion is a diffusion on two half-lines,  $(-\infty,0)$ and $(0,\infty)$, separated by a semi-permeable membrane located at~$0$. Therefore, it can be described by a Feller semigroup of operators in the space \[\cerf \] of continuous functions on 
\[ \R_\sharp\coloneqq  [-\infty,0-] \cup [0+,\infty] \]
where $0-$ and $0+$ are two points, representing 
positions to the immediate left and to the immediate right of the membrane; alternatively, members of $\cerf$ 
can be seen as continuous functions on $(-\infty,0)\cup(0,\infty)$ that have finite limits at $\pm \infty$ and one-sided finite limits at $0$. The membrane, in turn, is characterized by two non-negative parameters, say, $\alpha$ and $\beta$, describing its permeability for a particle diffusing from the left to the right and from the right to the left, respectively. More precisely, given such  $\al$ and $\be$, we define the generator of the snapping out Brownian motion as follows: it is the operator $A_{\al,\be}$ in $\x$  given by 
 \[ A_{\al,\be} f = f'', \]
on the domain composed of twice continuously differentiable $f \in \x $ such that $f', f''\in \x$ and 
\begin{align}\label{intro:bc}\begin{split}
f'(0-)&=\al(f(0+)-f(0-)), \\
f'(0+)&=\be(f(0+)-f(0-)). 
\end{split}\end{align}
See \cite{bobmor,lejayn} or \cite{knigaz}*{Chapters 4 and 11} for a more detailed description of the stochastic mechanism of filtering through the membrane, as governed by \eqref{intro:bc}, in terms of the celebrated L\'evy local time for Brownian motion.

As proved in \cite{tombatty} the operator $A_{\al,\be}$ generates not only a Feller semigroup but also a strongly continuous cosine family $\{\cosab(t),\, t \in \R\} $  of operators  in $\x$. This cosine family can in fact be given quite explicitly:  
\begin{equation} \cosab (t)f(x)  = \begin{cases} {C}(t) \wfl(x),  & x<0, \\
 {C}(t) \wfr(x),  & x>0, \\
 \end{cases}  \qquad  f\in \x, t \in \R,\label{intro:6} \end{equation}
where $C(t)$s are defined in \eqref{intro:0}, and $\wfl $ and $\wfr$ are certain extensions of the left and right parts of $f\in \x$ (see Figure \ref{rys3}). This formula for $\{\cosab(t),\, t \in \R\} $ is a proper counterpart of \eqref{intro:1} in the case of transmission conditions; in particular, it actually involves an extension operator and a restriction operator (see equation \eqref{ab:6}, further down).    
\begin{figure}
 \begin{tikzpicture}[scale=0.9]
 \draw [->] (-5,0) -- (5,0);
\draw [->] (0,-1.5)-- (0,2);
\draw[blue,semithick,domain=0:5, samples=600] plot (\x, {0.3+exp(-\x)*cos(deg(3*\x))});
\draw[dashed,thick,blue, domain=-5:0, samples=600] plot (\x, {0.3-3*exp(\x)+ 4*exp(2*\x)}); 
\draw [violet, semithick] plot [smooth] coordinates {(-5,-0.75) (-4,0.-0.73)  (-3,-0.25) (-2,0.75) (-0.7,0.1) (0,0.15)};
\draw [violet,thick, dashed]  plot [smooth] coordinates {(0,0.15) (1,1)  (2,0.8)  (3,1.5) (4,1)(5,1)};
\node [above] at (3,1.5) {$ \wfl $};
\node [above] at (-2,0.8) {$f$};
\node [below] at (-1,-0.3) {$ \wfr$};
\end{tikzpicture}
\caption{Two extensions of a single function $f \in \x$ (solid lines): extension $\wfl$ (violet) of its left part, and extension $\wfr$ (blue) of its right part.}\label{rys3} 
\end{figure}
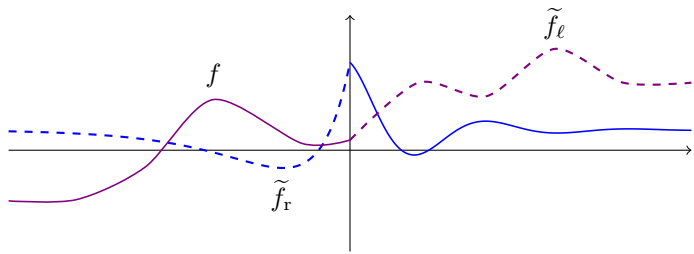

\subsection{Complementary transmission conditions: our first main result}\label{ctc}
Our main goal in this paper is to find an analogue of decomposition \eqref{intro:5} for the space 
\[\ce\coloneqq (\cer)^2\]
and  \emph{the basic Cartesian product cosine family} \[ \{\ced(t),\, t\in \R\} \] (`D' for `Descartes') defined in this space
by the formula 
\begin{equation}\label{intro:7} \ced (t) (f_1,f_2) = (C(t)f_1,C(t)f_2), \qquad t\in \R, f_1,f_2 \in \cer ,\end{equation}
where, to recall, $C(t)$s are introduced in \eqref{intro:0}. 
More precisely, we search for two subspaces of $\ce$ that are 
\begin{itemize} 
\item complementary to each other, and
\item related to transmission conditions;
\end{itemize}
the second requirement says in particular that each of these subspaces is invariant under $ \{\ced(t), \, t\in \R\}$. 

We discover that for this purpose transmission conditions \eqref{intro:bc} can be used and that they form a complementary pair with 
\begin{align} \label{intro:bc2} \begin{split}
f(0-)&=-f(0+), \\
f''(0+)&=\al f'(0-)+\be f'(0+), 
\end{split} \end{align}
as long as $\apb>0$. 
We prove in other words that 
\begin{enumerate} 
\item for any $f\in \mathbb X$, both extensions $\wfl$ and $\wfr$ of \eqref{intro:6} belong to $\cer$; the subspace $\y$ of $\ce$ composed of pairs of such extensions is invariant under $ \{\ced(t), \, t\in \R\}$ --- see Section \ref{exten}, 
\item there is a natural projection $\p$ on $\y$ --- see Section \ref{s:fa}, and  
\item the operator $\q \coloneqq I - \p$, where $I$ is the identity operator in $\ce$, projects on the invariant subspace $\z$ corresponding to the boundary condition \eqref{intro:bc2} --- see Section \ref{comp}.
\end{enumerate}
In summary:
\[ \ce = \y \oplus \z.\]

Concerning point (a) above, we remark that we provide an explicit formula for $\wfl$ and $\wfr$, and it is this formula that allows checking that $\wfr$ and $\wfl$ are members of $\cer$. In \cite{tombatty}, $\wfr$ and $\wfl$ are given merely implicitly and are proved to belong to a slightly larger space than $\cer$. Also, concerning (c), we note that this point requires proving in particular that the Laplace operator in $\x$ with domain characterized by transmission conditions \eqref{intro:bc2} is a cosine family generator and that the cosine family related to this operator is given by an abstract Kelvin formula analogous to \eqref{intro:6} --- this is done in Section \ref{gener}.

\subsection{Convergence to skew Brownian motion}
Sections \ref{aosb}--\ref{limit} are devoted to an approximation of the celebrated skew Brownian motion. As described in \cite{lejayskew} `The skew Brownian Motion appeared in the '70 in \cite{ito,walsh} as a natural generalization of the Brownian motion: it is a process that behaves like a Brownian motion except that the sign of each excursion is chosen using an independent Bernoulli random variable'. The paper by A.~Lejay cited above discusses a number of constructions of the skew Brownian motion that appeared in the literature since that time,  and various contexts in which this process and its generalizations are studied; see also \cite{yor97}*{p. 107}  and \cite{manyor}*{pp. 115--117}. It is well-known, for example, that the process can be obtained in the limit procedure of Friedlin and Wentzell's averaging principle \cite{fwbook} --- see \cite{fw}*{Thm. 5.1}, comp. \cite{emergence}*{Eq. (3.1)}. 

Much more recently, in \cite{abtk}, a link has been provided between skew Brownian motion and kinetic models of motion of a phonon involving an interface, of the type studied in
\cite{tomekkinetic,tomekkinetic3,tomekkinetic2,tomekkinetic1}, and the telegraph
process with elastic boundary at the origin \cite{wlosi,wlosi1}.

In Sections \ref{aosb}--\ref{limit}, we come back to the idea of \cite{tombatty} and \cite{knigaz}*{Chapter 11} that the skew Brownian motion can be obtained as a limit of snapping out Brownian motions. To explain this, let us replace $\alpha$ and $\beta$ in the transmission conditions \eqref{intro:bc} by $n\alpha$ and $n\beta$, respectively, and let $n \to \infty$. Heuristically, it is then clear that the limit transmission conditions should read \[ f(0+)=f(0-) \mquad{ and } \beta f'(0-)=\alpha f'(0+).\]
The first of these relations tells us that we should work with $f\in \x$ that are continuous at $x=0$, and the other is precisely the condition that characterizes the skew Brownian motion. Hence, we anticipate a theorem saying that a skew Brownian motion is a limit of snapping out Brownian motions, provided that permeability coefficients $\alpha$ and $\beta$ of the semi-permeable membrane converge to infinity whereas their ratio remains constant.

As we know from the references cited above, these intuitions can be transformed into a formal theorem saying that 
for any $s>0$, 
\begin{align}\label{intro:cosik} \gran \sup_{t\in [0,s]} \|\e^{t \anab}f - \e^{t\askew}f\| = 0, \qquad f \in \cer \end{align}
where  $\cer$ is seen as the subspace of $\x$ composed of functions $f$ such that $f(0-)=f(0+)$, $\{\e^{t\aab},\, t\ge 0 \}$ is the Feller semigroup describing snapping-out Brownian motion, and $ \{\e^{t\askew}, \,t\ge 0\}$ is that describing the skew Brownian motion. 

This formula deserves a closer scrutiny: it says that the skew Brownian motion is a limit of snapping out Brownian motions with permeability coefficients of the semi-permeable membrane increasing to infinity.  
Because of that, the skew Brownian motion should apparently describe the case in which the membrane is completely permeable. Nevertheless, as seen in the definition of $\askew$ and in its probabilistic description provided above, there remains an asymmetry between the way particles filter through zero when going from the left to the right and in the other direction.  Hence, the skew Brownian motion gains a plausible description as a process with a trace of semi-permeable membrane at $0$.

In Sections \ref{aosb}--\ref{limit} we exhibit a number of phenomena accompanying convergence \eqref{intro:cosik}:
\begin{itemize} 
\item First of all, on $\cer$ (but not outside of this space) not only the semigroups converge; so do also the related cosine families. Moreover, the limit is in fact uniform with respect to $t\in \R$ --- see Theorem \ref{thm:skew} (b) and Remark \ref{rem:one}.
\item For $f\not \in \cer $ the limit $\gran \e^{t\anab}f $ still exists for all $t>0$, but is merely uniform with respect to $t$ in compact subintervals of $(0,\infty)$ --- see Theorem \ref{thm:skew} (c).
\item As mentioned above, there are also semigroups and cosine families related to complementary transmission conditions \eqref{intro:bc2}.  If $\alpha$ and $\beta$ are replaced by $n\alpha$ and $n\beta$, respectively, and $n$ tends to $\infty$, these semigroups and cosine families converge, as well, and the limit is again uniform with respect to $t$ in the entire $\R$. Moreover, the limit semigroup and the limit cosine family turn out to be isomorphic copies of the semigroup and cosine family related to another skew Brownian motion, with the role of coefficients $\alpha$ and $\beta$ reversed  --- see Theorem \ref{thm:weks}. 
\item For the projections $\p$ described in Section \ref{ctc} and formally introduced in Section \ref{s:fa}, there is a strong  limit $\gran P_{n\alpha,n\beta}$. Furthermore, the limit operator turns out to be a projection on the invariant subspace of $\ce$ related to the skew Brownian motion --- see Section \ref{limit}.  
\end{itemize}

\subsection{Commonly used notation}\label{cun} 
\subsubsection{Constants $\alpha$ and $\beta$}
Throughout the paper we assume that $\alpha$ and $\beta$ are fixed non-negative constants. Also, to shorten formulae, we write 
\begin{align*} 
\sq\coloneqq \sqrt{2(\al^2+\be^2)}.
\end{align*}
Since the case of $\alpha=\beta=0$ is not interesting, in what follows we assume that ${\alpha+\beta}>0$, implying that $\gamma>0$ as well. 
\subsubsection{Limits at infinities, and transformations of functions} For $f\in \x$, we write 
\[f(\pm \infty)\coloneqq \lim_{x\to \pm \infty} f(x).\]
Moreover, we define
$ f^\es, f^e $ and $f^o$, also belonging to $\x$, by 
\begin{align} %\label{proj:3}  
f^\es (x)& \coloneqq  f(-x),   \qquad  x \in \R{\setminus \{0\}}, \nonumber  \end{align} 
and \begin{align} f^e  &\coloneqq \pol (f + f^\es), \qquad  f^o  \coloneqq \pol (f - f^\es ), \nonumber 
\end{align}so that \begin{align}
(f^\es)^\es &= f, \quad (f^o)^\es = (f^\es)^o= -f^o \mquad { and } (f^e)^\es =(f^\es)^e= f^e. \label{intro:T}
\end{align}
For $f $ in the linear space $\mathfrak C(\mathbb R)$ of all real continuous functions on $\mathbb R$, these definitions naturally extend to $x=0$ as well. % infinity (as in the section below, and in formula \eqref{intro:T1} in particular). 
In particular, for $e_a, a \in \mathbb R$, defined by \[e_a(x)=\e^{-a x}, \qquad x \in \mathbb R,\] we have $e_a^\es = e_{-a}$.  

We note that 
%if we know that $f,g\in \mathfrak C(\R)$ coincide on $\R^+$, then to prove that they are in fact equal to each other, it suffices to check that also $f^\es$ and $g^\es$ coincide on $\R^+$  (see e.g Lemmas \ref{lem:ab1} and \ref{lem:ab2}).  
%Furthermore, 
the basic cosine family, treated as a family of operators in $\mathfrak C(\R)$, commutes with the operations described above, that is, 
\begin{equation}\label{tran} (C(t)f)^o = C(t)f^o, \,   (C(t)f)^e = C(t)f^e, \, (C(t)f)^\es = C(t)f^\es, \,\,  f\in \mathfrak C(\R), t \in \R. \end{equation}

The following lemma will be used to establish % that certain $\phi \in \mathfrak C(\R)$ belong to $\cer. $
% Before presenting a corollary to Proposition \ref{propek}, we state a lemma to which we will refer throughout the paper to establish 
existence of a number of limits encountered in the paper.  

\begin{lem}\label{lopital} Let $f \in \cer$. Then, as long as $a>0$,
\begin{align*}
\lim_{x\to+\infty }  \int_0^x   \e^{-a (x-y)} f(y ) \ud y & = a^{-1}f(\infty),\nonumber \\
\lim_{x\to\pm\infty } \int_{-\infty}^x\e^{-a (x-y)}   f( y)  \ud y & = a^{-1}f(\pm\infty),\nonumber 
\\  \lim_{x\to\pm\infty } \int_x^\infty  \e^ {a (x-y)}  f (y) \ud y&= a^{-1}f (\pm\infty). 
%\label{intro:limit}
\end{align*}
\end{lem}
\begin{proof} This follows by l'Hospital's Rule. In fact, the first and the third relations here are consequences of the second.\end{proof}

\subsubsection{Laplace transform} By a slight abuse of notation, for $f\in \mathfrak C(\R)$, 
we write 
\[ \wh f (\lam ) \coloneqq  \int_0^\infty \e^{-\lam x} f(x)\ud x, \]
even though, in fact, the right-hand side here is not the Laplace transform of $f$ but of its restriction to the right half-axis. The integral featured above makes sense as long as, for example, there are $M$ and $\omega \ge0 $ such that $|f(x)|\le M\e^{\omega x}, x\ge 0$; then, $\wh f(\lam)$ is well-defined for $\lam >\omega$. 
%This will be the case in all examples considered in the paper. 

%We note that an $f\in \x $ (and, in particular, $f\in \cer$) is fully characterized by two Laplace transforms: $\wh f$ and $\wh {f^\es}$. This is because the Laplace transform, as restricted to continuous functions, is injective. 

\subsubsection{Convolution}
For $f,g\in \mathfrak C(\mathbb R)$, we write 
\[ f *g  (x) = \int_0^x f(x-y) g(y) \ud y , \qquad x \in \R;\]
we stress that, somewhat differently than in customary notation, this formula is valid for all $x\in \R.$ We have then $f*g=g*f$  and 
\begin{equation}\label{intro:T1} - (f *g)^\es = f^\es * g^\es . \end{equation}
Furthermore, $e_a, a \in \mathbb R$, introduced above, satisfy the Hilbert equation
\begin{equation}\label{hilbert} e_a - e_b = (b-a)e_a * e_b.\end{equation}
As a consequence, for 
 \[ \sinh_a (x) \coloneqq \sinh (a x),\quad  \cosh_a (x)\coloneqq \cosh (a x), \qquad x \in \R,\] 
we obtain
\begin{align} e_{\alpha+\beta} *[ (\alpha+\beta) \sinh_\gamma + \gamma \cosh_\gamma]&=\sinh_\gamma \nonumber\\  %\end{align}
 e_{\alpha+\beta}* [ \gamma \sinh_\gamma + (\alpha+\beta) \cosh_\gamma]&=\cosh_\gamma  - e_{\alpha+\beta}. \label{proj:dod}
 \end{align}
 
\subsubsection{The restriction operator}\label{tro}
The following restriction operator, denoted $R$ and mapping  $\ce$ onto the space $\x$ of Section \ref{isat}, will be of key importance in the entire paper. By definition $R$ assigns to a  pair $(f_1,f_2)\in \ce$ the member $f$ of $\x$ given by 
\[ f(x) = f_1(x), \text{ for } x<0  \mquad { and } f(x) = f_2(x), \text{ for } x >0. \]

\section{Invariant subspaces related to transmission conditions \eqref{intro:bc}}\label{exten}

We begin our analysis by connecting transmission conditions \eqref{intro:bc} with invariant subspaces of $\mathcal C$. To this end, in Section \ref{sec:doai}, we first find a family of subspaces of $\cer $ that are invariant under the basic cosine family \eqref{intro:0} (see Lemma \ref{lem:ab1}), and then use them to construct subspaces $\y$ of $\mathcal C$ that are invariant under the basic Cartesian product cosine family \eqref{intro:7}. In Section \ref{sec:rect} (see Thm. \ref{ab:genth} in particular) we specify in what sense the so-defined invariant subspaces are related to boundary conditions \eqref{intro:bc}: the cosine family generated by the operator $A_{\al,\be}$ introduced in Section \ref{isat} is an isomorphic image of the cosine family that is obtained by restricting the basic Cartesian product cosine family to $\y$. 

\subsection{Definition of an invariant space $\y$}\label{sec:doai}
Let $a>0$ be fixed, and let us think of the basic cosine family defined in \eqref{intro:0} as composed of operators acting in $\mathfrak C(\mathbb R)$ (the linear space of real continuous functions on $\mathbb R$).  
We start by noting the following formula, which can be proved by direct calculation, and holds for all $t,x\in \R$ and $\phi \in \mathfrak C(\R)$:
\begin{equation}\label{ab:1} C(t) (e_a * \phi) (x) = (e_a * C(t)\phi)(x) + \e^{-ax} (e_a*\phi)^e (t). \end{equation}
This formula reveals that functions $\phi$ that satisfy  $(e_a*\phi )^e=0$ play a~special role in $\mathfrak C(\R)$. In our first lemma, we characterize such functions in more detail. 

\begin{lem}\label{lem:ab1} For $\phi \in \mathfrak C(\R)$ the following conditions are equivalent. 
\begin{enumerate} 
\item $\phi^o= a e_a *\phi$ on $\R^+\coloneqq [0,\infty)$;
\item $\phi^o= a e_a *\phi$ on $\R$;
\item $e_a * \phi$ is odd, that is, $(e_a*\phi)^e = 0$ on $\R$; 
\item For all $t\in \R$, $ C(t) (e_a * \phi) = e_a * C(t)\phi$ on $\R$. 
\end{enumerate}
 \end{lem}
\begin{proof} We recall that two members, say, $\phi_1$ and $\phi_2$, of $\mathfrak C(\R)$ are equal iff simultaneously  $\phi_1=\phi_2$ on $\R^+$ and $\phi_1^\es = \phi_2^\es$ on $\R^+$.  

Suppose now that (a) is true. To see that (b) holds also, we need to check, by \eqref{intro:T1} and the remark made above, that  $\phi^o= a e_a^\es *\phi^\es$ on $\R^+$, that is, by (a), that $e_a^\es * \phi^\es = e_a * \phi$ on $\R^+.$ 
On the other hand, using assumption (a) again, we see that 
\[e_a^\es *\phi^\es = e_{-a} *(\phi- 2ae_a*\phi) = (e_{-a} - 2a e_{-a}* e_a) * \phi, \]
and the last expression equals $e_a*\phi$ by the Hilbert equation \eqref{hilbert}. This completes the proof of 
(a)$\implies$(b); the converse is obvious. 

If (b) holds, $e_a * \phi$ must be odd because so is $\phi^o$; this shows (b)$\implies$(c). Conversely, (c) means that $-e_a * \phi =(e_a*\phi)^\es $ and this, by \eqref{intro:T1}, can be expanded as 
\( \int_0^x \e^{-a(x-y)}\phi(y) \ud y = \int_0^x \e^{a(x-y)}\phi(-y)\ud y, x \in \R. \) Differentiating this relation with respect to  $x$ yields (b). Finally, (c) is equivalent to (d) by \eqref{ab:1}.
 \end{proof}

Our lemma has an immediate bearing on the following subspace of $\ce$: 
\begin{align*} \mathsf C_{a,b} & \coloneqq \{(\phi_1,\phi_2) \in \ce: \phi_1^o = ae_a*\phi_1; \phi_2^o=b e_a*\phi_1\},\end{align*}
where $b>0$ is another constant. Namely, we have the following corollary. 

\begin{cor}The subspace $\mathsf C_{a,b}$ is invariant under the basic Cartesian product cosine family. \label{cor:ab1} \end{cor}
\begin{proof} Since $\ced (t) (\phi_1, \phi_2)= (C(t)\phi_1,C(t)\phi_2)$, we need to prove first of all that $(C(t)\phi_1)^o= ae_a*(C(t)\phi_1)$ for all $(\phi_1,\phi_2)\in \mathsf C_{a,b}$ and $t\in \R$. But we know that $\phi^o_1=ae_a*\phi_1$, and this tells us that  $\phi_1$ satisfies condition (b) in Lemma \ref{lem:ab1}. Hence, $\phi_1$  satisfies all the other conditions listed in Lemma \eqref{lem:ab1}, and in particular we can use (d). By the first relation in \eqref{tran}, this renders
$(C(t)\phi_1)^o=C(t)\phi^o _1= aC(t)(e_a *\phi_1)= ae_a *(C(t)\phi_1)$, thus completing the first part of the proof. 

We are left with showing that $(C(t)\phi_2)^o=b e_a*(C(t)\phi_1)$. To this end, we calculate as above, using the fact that $\phi_2^o=b e_a*\phi_1$: $(C(t)\phi_2)^o=C(t)\phi_2^o= b C(t)(e_a*\phi_1)= b e_a *(C(t)\phi_1), $ as desired. 
\end{proof}

To continue, we note that any real matrix $M=\begin{pmatrix} m_{1,1} & m_{1,2} \\ m_{2,1} & m_{2,2} \end{pmatrix}$ induces a~bounded linear operator in $\ce$, also denoted $M$, by the formula
\begin{equation}\label{ab:2} M \binom{\phi_1}{\phi_2} \coloneqq \begin{pmatrix} m_{1,1}\phi_1^\es  + m_{1,2}\phi_2^\es \\ m_{2,1}\phi_1+ m_{2,2}\phi_2 \end{pmatrix}.\end{equation}
By the third relation in \eqref{tran}, this operator commutes with the basic Cartesian product cosine family. It follows that the image of $\mathsf C_{a,b}$ via $M$ is also an invariant subspace of $\ce$ (under $ \{\ced(t),\, t\in \R\}$). Specializing to (see Section \ref{cun}) $a\coloneqq {\alpha+\beta}, b\coloneqq \pol \gamma^2$ and 
\begin{equation} \label{ab:2.5} M \coloneqq \frac 1{\alpha-\beta} \begin{pmatrix} \beta & -1 \\ \alpha& -1 \end{pmatrix}, \end{equation}
we obtain the invariant subspace 
\begin{equation}\label{ab:3} \y \coloneqq M \left (\mathsf C_{{\alpha+\beta},{\frac {\gamma^2}2}} \right ).\end{equation}

This definition, of course, is meaningless if $\alpha=\beta$; and so in this case we proceed differently. Namely, we introduce 
\[ \mathsf C_{\alpha+\beta}^\sharp \coloneqq \{(\phi_1,\phi_2)\in \ce: \phi_1^o = (\alpha+\beta) e_{\alpha+\beta} * \phi_1, \phi_2^o=0\},\]
and note that this subspace is invariant under $\ced$, because of Lemma \ref{lem:ab1}  and the first relation in \eqref{tran}. Then, we define $\mathcal C_{\alpha,\alpha}  $ as the image of this subspace via the operator $M^\sharp$ of the form \eqref{ab:2}: 
 \begin{equation}\label{ab:4} \mathcal C_{\alpha,\alpha}  \coloneqq M^\sharp (\mathsf C_{\alpha+\beta}^\sharp),\end{equation}
where 
 \begin{equation}\label{ab:4.5} M^\sharp \binom{\phi_1}{\phi_2} =\pol \binom{\phi_2^\es - \phi_1^\es}{\phi_1 + \phi_2}.
 \end{equation}
\begin{lem} \label{lem:ab2}For $(\psi_1,\psi_2)\in \ce$ the following conditions are equivalent. 
\begin{enumerate} 
\item $(\psi_1,\psi_2)$ belongs to $\y$;
\item $\psi_1^o=\alpha  e_{\alpha+\beta} *(\psi_2 - \psi_1^\es)$ and $\psi_2^o=\beta  e_{\alpha+\beta} *(\psi_2 - \psi_1^\es)$
on $\R$;
\item $\psi_1^o=\alpha  e_{\alpha+\beta} *(\psi_2 - \psi_1^\es)$ and $\psi_2^o=\beta  e_{\alpha+\beta} *(\psi_2 - \psi_1^\es)$
on $\R^+$. 
\end{enumerate}
\end{lem}

\begin{proof}  \textbf {(a)$\iff$(b). Case $\alpha\not =\beta$.}
 By definition, $(\psi_1,\psi_2)$ belongs to $\y$ iff there are $\phi_1,\phi_2$ such that 
\[ \phi_1^o= (\alpha+\beta) e_{\alpha+\beta} * \phi_1, \qquad \phi_2^o ={\textstyle \frac{\gamma^2}2} e_{\alpha+\beta} * \phi_1 \]
and 
\begin{equation}\label{added} \binom{\psi_1}{\psi_2} = M \binom{\phi_1}{\phi_2} \quad \text{or, equivalently,} \quad
\binom{\phi_1}{\phi_2} = M^{-1} \binom{\psi_1}{\psi_2} \coloneqq \begin{pmatrix} \psi_2 -  \psi_1^\es\\ \beta \psi_2 - \alpha \psi_1^\es \end{pmatrix}.\end{equation}
Hence, (a) implies
\begin{align*} \binom{\psi_1^o}{\psi_2^o} &= \frac1{\alpha-\beta} \binom{\phi_2^o-\beta \phi_1^o}{\alpha \phi_1^o-\phi_2^o} = \frac1{2(\alpha-\beta)} \binom{[\gamma^2-2\beta(\alpha+\beta)] e_{\alpha+\beta} *\phi_1}{[2\alpha(\alpha+\beta)-\gamma^2] e_{\alpha+\beta} *\phi_1}\\
&= \binom{\alpha e_{\alpha+\beta} *(\psi_2 - \psi_1^\es)}{\beta e_{\alpha+\beta} *(\psi_2 - \psi_1^\es)},
 \end{align*}
which is (b). The converse is proved analogously. 

\textbf {(a)$\iff$(b). Case $\alpha=\beta$.} We proceed similarly: $(\psi_1,\psi_2) $ belongs to $\mathcal C_{\alpha,\alpha}$ if there is $(\phi_1,\phi_2) \in \ce$ such that  
\[ \phi_1^o = (\alpha+\beta) e_{\alpha+\beta} * \phi_1 \mquad{ and }  \phi_2^o=0\]
and 
\[ \binom{\psi_1}{\psi_2} = M^\sharp \binom{\phi_1}{\phi_2} \, \, \text{or, equivalently,}\, \, 
\binom{\phi_1}{\phi_2} = (M^\sharp)^{-1} \binom{\psi_1}{\psi_2} \coloneqq \begin{pmatrix} \psi_2 -  \psi_1^\es\\ \psi_2 + \psi_1^\es \end{pmatrix}.\] 
We omit the details. 

\textbf{(b)$\iff$(c).} Since (b) obviously implies (c), we are left with showing (c)$\implies$(b). 
%To this end, we note that, if (c) holds, we only need to check, by \eqref{intro:T1} and the remark made at the beginning of the proof of Lemma \ref{lem:ab1}, that $\psi_1^o=\alpha e_{\alpha+\beta}^\es * (\psi_2 - \psi_1^\es)^\es$ and  $\psi_2^o=\beta e_{\alpha+\beta}^\es * (\psi_2 - \psi_1^\es)^\es$ on $\R^+.$ This, in view of (c), amounts to showing that 
%\begin{equation}\label{ab:5} 
%e_{\alpha+\beta} * (\psi_2 - \psi_1^\es)
%=e_{\alpha+\beta}^T  * (\psi_2 - \psi_1^\es)^\es \qquad \text{ on }\R^+. 
%\end{equation} 
%On the other hand, for $\phi \coloneqq \psi_2 - \psi_1^\es,$ (c) implies $\phi^o = (\alpha+\beta) e_{\alpha+\beta}*\phi $ on $\R^+,$ and this, by Lemma \ref{lem:ab1} proves that $e_{\alpha+\beta} * \phi = e_{\alpha+\beta}^\es * \phi^\es$, establishing \eqref{ab:5}. 
To this end, we note that $\psi^o_1$ and $\psi_2^o$ are odd and thus it suffices to show that so is $e_{\alpha + \beta} * (\psi_2 - \psi_1^\es)$ if (c) holds. On the other hand, for $\phi \coloneqq \psi_2 - \psi_1^\es,$ (c) implies $\phi^o = (\alpha+\beta) e_{\alpha+\beta}*\phi $ on $\R^+,$ and this, by Lemma \ref{lem:ab1} proves that $e_{\alpha+\beta} * \phi$ is odd. \end{proof}

\subsection{Relation of $\y$ to transmission conditions \eqref{intro:bc}}\label{sec:rect} 
\begin{center}
\begin{figure} 
\includegraphics[scale=0.5]{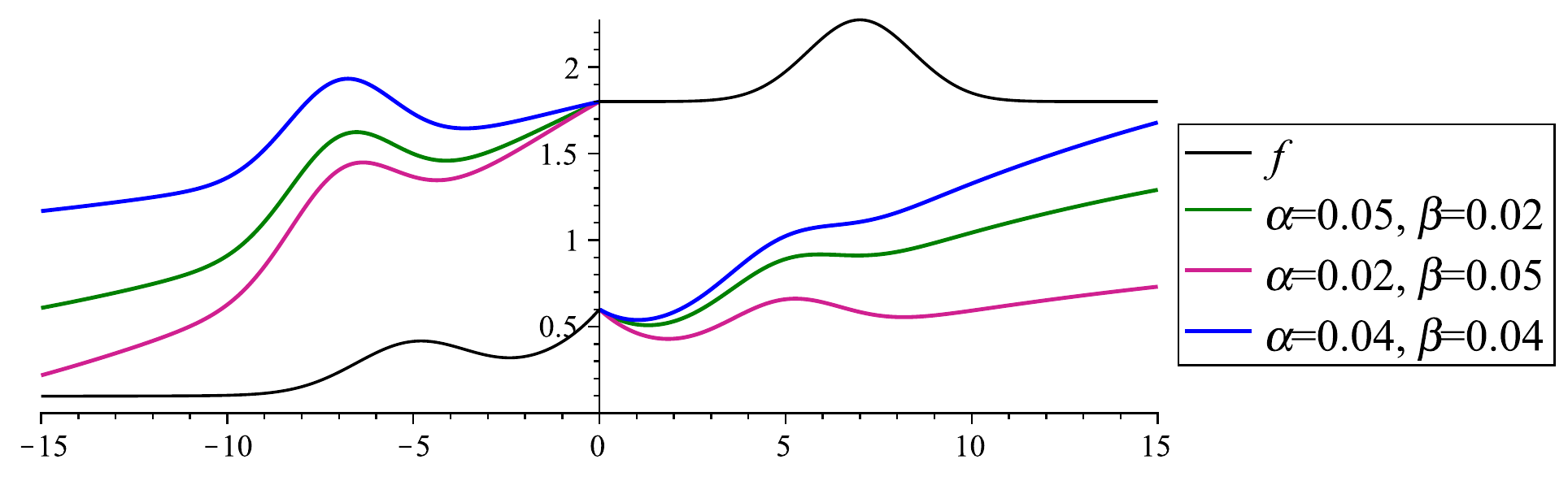}
\caption{Extensions  $E_{\al,\be}f$ of a single function $f \in \x$.
 }

\label{rys_ext}
\end{figure}
\end{center}

\vspace{-0.5cm}
We are now ready to exhibit connection between $\y$ and transmission conditions \eqref{intro:bc}. To this end, 
given $f\in \x$ we define extensions $\wfl,\wfr\in \mathfrak C(\R)$ of its left and right parts (see Figures \ref{rys3} and \ref{rys_ext}) by 
\begin{align}
\wfl (-x) &= f(-x),  \qquad \wfr (x) = f(x),
 \nonumber \\
\wfl (x) &= f(-x)+2\al  e_{\su}* [f-f^\es] (x),  \nonumber %\label{proj:0} 
\\
\wfr (-x)&=f(x)- 2\be  e_\su  * [f-f^\es](x) , \qquad  x> 0, \label{proj:1} \end{align}
and $  \wfr (0)=f(0+), \wfl (0)=f(0-).$ 

%\begin{prop}\end{prop}

\begin{prop}\label{propek1} The transformation $E_{\al,\be}$:
\(  \x \ni f \mapsto (\wfl,\wfr) \in [\mathfrak C(\R)]^2 \)
is an isomorphism of $\x$ and $\y$.\end{prop}
%\begin{prop}The transformation $E$:
%\[ \x \ni f\mapsto (\wfl,\wfr) \in [\mathfrak C(\R)]^2 \] 
% is an isomorphism of $\x$ and $\y$, as long as ${\alpha+\beta} >0$. 
%\end{prop}

\begin{proof} Our first claim is that $\wfl$ and $\wfr$ are members of $\cer$. Since, by definition of $\x$, existence of 
 $\lim_{x\to \infty} \wfr (x)$ and $\lim_{x\to -\infty} \wfl (x)$ is assumed, this follows by Lemma \ref{lopital}: in view of \eqref{proj:1}, the lemma renders 
\(\gra \wfl (x) =  \frac{2\alpha }{\alpha+\beta}f(\infty)  +\frac{\beta-\alpha }{\alpha+\beta} f(-\infty)
\) and $\lim_{x\to -\infty} \wfr (x)= \frac{\alpha -\beta}{\alpha+\beta} f(\infty)+ \frac{2\beta}{\alpha+\beta} f(-\infty).$ 

Next, relations \eqref{proj:1} say that on $\R^+$ we have 
\[ \wfl^o = \alpha e_{\alpha+\beta} * (\wfr - \wfl^\es) \mquad {and }  \wfr^o = \beta e_{\alpha+\beta} * (\wfr - \wfl^\es) \]
and this, when combined with the fact established above,  by Lemma \ref{lem:ab2} proves that the pair $(\wfl,\wfr)$ belongs to $\y$. 

The extension operator $E_{\al,\be}\colon\x \to \y$ is clearly linear and bounded with \begin{equation}\label{ab:5a} \|E_{\al,\be}\|\le 1 +\textstyle{\frac {4}{\alpha+\beta}}\max (\alpha,\beta)\le 5.\end{equation} It is also surjective, 
because it has its right inverse --- the operator
$R$ (introduced in Section \ref{tro}) restricted to $\y$. (This is just a restatement of the fact that, by (c) in Lemma \ref{lem:ab2}, a pair $(\psi_1,\psi_2)$ in $\y$ is determined by values of $\psi_1$ on the negative half-axis together with the values of $\psi_2$ on the positive half-axis. For instance, $\psi_1(x) = \psi_1(-x) + 2 \alpha \int_0^x \e^{-(\alpha+\beta)(x-y)}[\psi_2(y)- \psi_1(-y)]\ud y$, $x>0$.) 
Since injectivity  of $E_{\al,\be}$ is beyond doubt, we are done. 
\end{proof}

Invariance of $\y$ (established in Corollary \ref{cor:ab1}) together with the fact that $E_{\al,\be}$ is an isomorphism allows defining the following family of operators in $\x$:
\begin{equation}\label{ab:6} \cosab (t) \coloneqq R \ced (t) E_{\al,\be}, \qquad t \in \R. \end{equation}
It is easy to see that $\{\cosab (t), \,t \in \R\}$ is a strongly continuous cosine family; it is in fact an isomorphic image of the Cartesian product basic cosine family, restricted to the invariant subspace $\y$. Our last proposition in this section says that the generator of this family is a Laplace operator with domain characterized by transmission conditions \eqref{intro:bc}.
%\begin{prop} \end{prop}
\begin{thm} \label{ab:genth} The generator  $A_{\al,\be}$ of the snapping out  Brownian motion introduced in Section \ref{isat} is the generator of the cosine family defined by the abstract Kelvin formula \eqref{ab:6}. \end{thm}
\begin{proof}\ \ 

\textbf{Step 1.} 
Let $G$ be the generator of $\{\cosab, \,t \in \R\}$; we claim that $D(A_{\al,\be})=D(G)$. To this end, we note that, by definition, an $f\in \x$ belongs do $\dom{G}$ iff the pair $(\wfl,\wfr)$ belongs do the domain of the generator of $\{\ced (t)_{|\y}, t \in \R\},$ that is, if $\wfl$ and $\wfr$ are twice continuously differentiable on the entire $\R$ with $\wfl'',\wfr''\in \cer.$ This is the case if $f$ is twice continuously differentiable with $f',f'' \in \x$ and both extensions have derivatives of second order at $0$. Hence, the crux of the matter is to show that under such circumstances these latter derivatives exist iff $f$ satisfies  transmission conditions \eqref{intro:bc}.  

Assume thus that $f\in \x$ is such that $\wfl$ and $\wfr $ are continuously differentiable on the entire $\R$. Using \eqref{proj:1}, we can write out the formula for the difference quotient, to see that the right-hand derivative of $\wfl$ at $0$ is minus the left-hand derivative of $\wfl$ at $0$ plus $2\alpha (f(0+) - f(0-))$. Since the one sided derivatives coincide by assumption, we thus obtain $(\wfl)'(0)= \alpha (f(0+) - f(0-)).$  Assumed continuity of $(\wfl)'$ together with the fact that the derivatives of $\wfl$ and $f$ coincide on $(-\infty,0)$ leads to the conclusion that $f'(0-)=\alpha (f(0+) - f(0-))$. The other condition in \eqref{intro:bc} is checked similarly.   

To prove the converse, assume that $f$ belongs to $D(A_{\al,\be})$. Then, existence of the left-hand limit of $f$ and $f'$ at $0$ together with l'Hospital's Rule shows that the left-hand derivative of $\wfl$ at this point exists and coincides with $f'(0-).$ Moreover, writing out again the formula for the difference quotient and using l'Hospital's Rule, we check that the right-hand derivative of $\wfl$ at zero exists and equals $-f'(0-)+2\alpha (f(0+) - f(0-))$. By assumption \eqref{intro:bc}, it follows that the one-sided derivatives of $\wfl$ coincide and equal $\alpha (f(0+) - f(0-))$. Next, we note that for $x>0$
\begin{align*} (\wfl)'(x) &= -f'(x) - 2\alpha (\alpha+\beta) e_{\alpha+\beta}* [f - f^\es] (x) + 2\alpha [f-f^\es](x)\end{align*}
and argue similarly as above, using \eqref{intro:bc} and l'Hospital's Rule, to see that one-sided second order derivatives of $\wfl$ at $0$ exist and coincide. The case of $\wfr$ is analogous. 

\textbf{Step 2.} Let now $f\in D(G)=D(A_{\al,\be})$. By Step 1, $E_{\al,\be}f=(\wfl,\wfr)$ is then a pair of twice continuously differentiable functions and thus we have 
\[ \lim_{t\to 0}2t^{-2} (\ced (t)E_{\al,\be}f - E_{\al,\be}f) = ((\wfl)'',(\wfr)'').\]
It follows that $Gf = R((\wfl)'',(\wfr)'')= f''=A_{\al,\be}f$, completing the proof. \end{proof}

\section{Natural projections on the space $\y$ and its complement}\label{proj} 
In Section \ref{exten} we introduced $\y\subset \mathcal C$ as a subspace that is invariant under the basic Cartesian product family, and then proved that it can be seen as the space of extensions of members of $\x$ (the extensions given by \eqref{proj:1}), thus establishing a link between $\y$ and the transmission conditions  \eqref{intro:bc}. In this section we find a natural projection on the space $\y$.

 First, in Section \ref{dotp}, we use a somewhat heuristic argument  to \emph{derive} a candidate for such a projection by using $L^2$ intuitions. Then, in Sections \ref{s:fa} and \ref{comp} we check that the formula guessed in Section \ref{dotp} indeed defines a projection operator $\p: \ce\to \y$. In the course of this analysis we also characterize the complement of $\y$, that is, the space $\z$ described in Section \ref{comp}. The main result is summarized in Theorem \ref{ttrzy}.

%In terms of $\sq$ and $\su$ of Section \ref{cun}, the space $\y$ on which we will project is defined as follows. 
To recall, by Lemma \ref{lem:ab2},  $\y$ is the space of $(\gle,\gre) \in \ce$ satisfying
\begin{align}
\gle^o  =\al \e_\su* (\gre-\gle^\es),   \qquad
\gre^o=\be  \e_\su *(\gre-\gle^\es). \label{proj:1aa}
 \end{align} 
%In other words, for the Laplace transforms of coordinates of $(g_1,g_2) \in \y$ we have  \begin{align} \wh g_1 -  \wh {g_1^\es} &= 2\alpha \wh e_\su (\wh g_2 - \wh {g_1^\es}), \nonumber \\  \wh g_2 - \wh {g_2^\es}&= 2\beta  \wh e_\su (\wh g_2 - \wh {g_1^\es}). \label{proj:1a} \end{align}

\subsection{Heuristic derivation of projection }\label{dotp}
Led by the conviction that the most natural projections are orthogonal projections in a Hilbert space, 
given $(f_1,f_2) \in \ce$  
we fix $n \in \N$ and search for a pair $(\gle,\gre) \in \y$  that minimizes the functional 
\[ L(\gle,\gre) \coloneqq \int_{-n}^{n} [\gle(x)-\fle(x)]^2+[\gre(x)-\fre(x)]^2 \ud x ;\] 
the reason why we consider the integral over a finite interval, is that functions $\gle$ and $\gre$  need not belong to $L^2(\R)$. Using \eqref{proj:1aa} we obtain that $L(\gle,\gre)$ is equal to
\begin{align*}
& \int_0^n \bigg ( \Big[\gle^\es(x) - \fle^\es(x)\Big]^2+\Big[ \gle^\es(x) + 2\al  \e_\su * (\gre-\gle^\es)(x)- \fle(x)\Big]^2 
\\&\phantom{=} +\big[\gre(x) - \fre(x)\big]^2+ \Big[ \gre(x) - 2\be  \e_\su * (\gre-\gle^\es)(x) -\fre^\es(x)\Big]^2  \bigg) \mud x .
\end{align*}
In terms of functions $\fh_1$ and  $\fh_2$ defined by
\begin{align*}
{h}_1 & \coloneqq e_\su * g_1^\es , \qquad h_2 \coloneqq  e_\su * g_2, \end{align*}
and related to $g_1^\es$ and $g_2$ also by
\begin{align}
 \fh_1' &= g_1^\es - (\su) \fh_1, \qquad \fh_2' = g_2 - (\su) \fh_2,  \label{der_iden}
 \end{align}
 the quantity $L(\gle,\gre)$ can be written as
\begin{align} 
  &  \int_0^n \bigg (\big[ \fh_1'(x) +(\su) \fh_1(x) - \fle^\es(x)\big]^2 + \big[ \fh_2'(x) +(\su) \fh_2(x) - \fre(x)\big]^2 \nonumber \\ \nonumber 
 &\phantom{==}+\big[ \fh_1'(x) - (\al-\be) \fh_1(x) +2\al \fh_2(x)- \fle(x)\big]^2  \\
 &\phantom{==}+ \big[\fh_2'(x) +2\be \fh_1(x) + (\al-\be) \fh_2(x) - \fre^\es(x)\big]^2 \bigg ) \mud x \label{proj:2} .
\end{align}
Therefore, the minimizing $(\fh_1,\fh_2)$  has to satisfy the Euler--Lagrange equations
\[ \frac {\partial K}{\partial \fh_i} =   \frac {\partial^2 K}{\partial x \partial \fh_i'},  \qquad i=1,2, \] where $K(\fh_i,\fh_i',x)$, $i=1,2$, is the integrand in \eqref{proj:2}. 
If we assume additionally that $\fle$ and $\fre$ are continuously differentiable,
this leads to the following system of ODEs with constant coefficients for $(\fh_1,\fh_2)$:
\begin{align*} 
h_1''+ (\al-\be)\fh_2'  -(\al^2+3\be^2)\fh_1 +(\al-\be)^2 \fh_1&=(\fle^e)' +\al \fle^o-\be(\fle^e+ \fre^\es ) , \\ 
  h_2''- (\al-\be)\fh_1'  +(\al-\be)^2 \fh_1 -(3\al^2+\be^2)\fh_2 &=(\fre^e)'-\be \fre^o -\al(\fre^e+\fle ),
\end{align*}
with initial conditions $  \fh_i(0)=0, \fh_i'(0)=g_i(0),$ $i=1,2$. Solving this system and  using relations \eqref{der_iden}
 we find the formula for $(\gle,\gre)$ that minimizes $L$. Namely, a somewhat long calculation shows that on the interval $[0,n]$,  
\begin{align} 
\gle &= \fle^e + c \big (\textstyle{\frac{2\al}{\sq}}\sinh_\sq - \cosh_\sq  \hspace{-2pt}\big) - \big(\textstyle{\frac{2\al}{\gamma}} k_1-\textstyle{\frac{\gamma}{2}}  k_2\big ) * \sinh_\sq   +( k_1-\al  k_2) *  \cosh_\sq,\nonumber \\    
\gre & = \fre^e  + c \big (\textstyle{\frac{2\be}{\sq}}\sinh_\sq + \cosh_\sq  \hspace{-2pt}\big )-\big(\textstyle{\frac{2\be}{\gamma}} k_1+\textstyle{\frac{\gamma}{2}} k_2\big) *\sinh_\sq  - ( k_1+\be  k_2)*\cosh _\sq,  \label{proj:gr}
 \end{align}
where  $c=c(n)$ is a real constant, and
\begin{align}
 k_1 \coloneqq   \alpha f_1^o + \beta f_2^o, \qquad
 k_2\coloneqq f_1^e - f_2^e . \label{proj:F} \end{align}

To summarize: functions $g_1$ and $g_2$ defined by \eqref{proj:gr} and \eqref{proj:F} form a pair that is a
candidate for a projection of $(f_1,f_2)$ on $\y$. However, so far they are defined merely on $[0,n]$ and, since $c$ depends on $n$, a priori we cannot assume that as $n$ increases, formula \eqref{proj:gr} still  defines the same functions. In fact, our $L^2$-based argument has not determined $c= c(n)$ as yet.  Our next step, therefore, is a leap of faith: we assume that $c$ does not depend on $n$, so that \eqref{proj:gr} is a consistent definition on the entire right half-axis. Furthermore, we note that, by \eqref{proj:1}, for $(g_1,g_2) \in \y$, $(g_1)_{|[0,\infty)}$ is determined by $(g_1)_{|(-\infty,0]}$ and $(g_2)_{|[0,\infty)}$, and a similar remark applies to $(g_2)_{|(-\infty,0]}$. In fact, as in \cite{bobrat}*{Proposition 2.1 (c)}, we conjecture  that \eqref{proj:gr} works on the entire $\R$.  

To complete our search for $g_1$ and $g_2$, we should define $c$. To this end, we recall that for a function $\phi\colon[0,\infty)\to \R$ a finite limit $\gra \e^{\gamma x} \phi (x)$ cannot exist unless 
$\gra \phi (x)=0$ (because, by assumption, $\gamma >0$). It follows that, for $\gra g_2(x) $ to be finite, it is necessary  that $c$ in  \eqref{proj:gr} be given by (see Section \ref{cun} for the notation we use here)
\begin{align}\label{proj:C2} \begin{split}
c&\coloneqq \wh  k_1(\sq)+\textstyle{\frac{\sq}{2}}\wh{ k_2}(\sq).
\end{split} \end{align}
In the next section, we show that formulae \eqref{proj:gr} and \eqref{proj:F} complemented by this necessary condition for existence of $\gra g_2(x) $ define a pair of members of $\cer $, and in fact, the map $(f_1,f_2) \mapsto (g_1,g_2) $ is a projection on $\y$. %Below, we prove that this somewhat holey argument leads to a genuine projection on $\y$. 
\subsection{Formal analysis}\label{s:fa} 

In this section we study the map 
\[ \ce \ni (f_1,f_2) \overset {\p} \longmapsto (g_1,g_2)\]
where $g_1$ and $g_2$ are given by 
\eqref{proj:gr} (on the entire $\R$) supplemented by \eqref{proj:F}--\eqref{proj:C2}.
Our ultimate aim (achieved in Section \ref{comp}) is showing that $\p$ is a projection on $\y$, but in this section we content ourselves with proving that the range of $\p$ is $\y$.

\begin{prop}\label{raz}
For $(\fle,\fre)\in \ce$, $\p(\fle,\fre)$ belongs to $\y$. 
\end{prop}

\begin{proof} \ 

{\textbf {Step 1.}}  
 We start with proving that for $(\gle,\gre)$ defined by \eqref{proj:gr}--\eqref{proj:F} condition \eqref{proj:1aa} is satisfied (regardless of the choice of $c$).
Combining the definition of $\p$ and identities \eqref{intro:T}, \eqref{intro:T1} with the facts that $ k_1$ is odd and $ k_2$ is even,
we obtain
 \begin{align}
  g_1^o  =\textstyle{\frac{2\alpha c}{ \gamma}}{\sinh_\gamma}- \textstyle{\frac{2\al}{\sq}} k_1* {\sinh_\gamma} -\al  k_2 * {\cosh_\gamma}, \nonumber \\
   g_2^o  = \textstyle{\frac{2\beta c }{\gamma}}{\sinh_\gamma} -  \textstyle{\frac{2\be}{\sq}} k_1* {\sinh_\gamma} -\be k_2* {\cosh_\gamma}.
 \end{align}
 Thus $\beta   g_1^o =\alpha g_2^o$ and we are left with showing that 
 \[   g_1^o = \alpha {e_{\alpha+\beta}} *( g_2 - {g_1^\es}) .\]
Again, by  \eqref{proj:gr}, \eqref{intro:T} and \eqref{intro:T1}, 
\begin{align*}   g_2 - {g_1^\es}  & =  -k_2 + 2c\big( \textstyle{\frac{\alpha+\beta} {\gamma}} {\sinh_\gamma} +{\cosh_\gamma}\big) - \big(\textstyle{\frac{2(\al+\be)}{\sq}} k_1+\sq  k_2   \big)*\sinh_\gamma    \\
&\phantom{=}- (  2  k_1+(\su)  k_2 )*{\cosh_\gamma}.\end{align*}
Therefore, by the first relation in \eqref{proj:dod}, it suffices to show that 
\begin{align*}
\textstyle{\frac 2 \gamma }{ k_1}* [{\sinh_\gamma} &- {e_{\alpha+\beta}}  *( (\alpha+\beta) {\sinh_\gamma} +\gamma {\cosh_\gamma}) ]
\\& =  { k_2}* [ {e_{\alpha+\beta}}*(\gamma  {\sinh_\gamma} + (\alpha+\beta)  {\cosh_\gamma})-{\cosh_\gamma} +e_{\alpha+\beta}]. \end{align*}
(We stress that this argument works regardless of the choice of constant $c$.) 
Since, by \eqref{proj:dod}, both expressions in brackets vanish, proof of the first step is completed.

{\textbf {Step 2.}} Here, we show that the limits $\lim_{x\to \pm \infty} g_{i}( x)$, $i=1,2$, exist and are finite. 
Starting with $\gra g_1(x)$, we write, by the  definition of $c$ given in \eqref{proj:C2}, 
\begin{align}
g_1(x) & = \fle^e(x) +\pol(\gamma+2\al) \int_{-\infty}^x \big[\textstyle{\frac 1\gamma} k_1(y)-\pol  k_2(y) \big]\e^ {-\sq  (x-y)}   \ud y \nonumber \\
 &\phantom{=} -\pol(\gamma-2\al)  \int_x^{\infty} \big[\textstyle{\frac 1\gamma} k_1(y)+\pol  k_2(y) \big]\e^ {\sq (x-y)}   \ud y , \qquad x\in \R.\label{intro:eqdef_g1}  \end{align}
Since $f_2^e,  k_1$ and $ k_2$ all belong to $\cer$, the existence and finiteness of  $\lim_{x\to \pm \infty}  g_1(x)$  follows by Lemma \ref{lopital}.

Turning to $\lim_{x\to -\infty} g_2(x)$, we observe that 
\begin{align}
g_{2}(x) &= \fre^e(x) +\pol (\gamma+2\be) \int_x^\infty  \big[\textstyle{\frac 1 \gamma} k_1(y)+\pol  k_2(y) \big]\e^{\sq (x-y)}   \ud y \nonumber \\   \label{intro:eqdef_g2}
 &\phantom{=} -\pol (\gamma-2\be ) \int_{-\infty}^x \big[\textstyle{\frac 1 \gamma} k_1(y)-\pol  k_2(y) \big]\e^{-\sq (x-y)}   \ud y , \qquad x\in \R.   
 \end{align}
and, as with the previous limit, we deduce existence of $\lim_{x\to \pm \infty} g_2(x)$ from Lemma \ref{lopital}.
  \end{proof}

It is tempting to use arguments similar to these presented above to show that for $(f_1,f_2)\in \y$, $\p (f_1,f_2) $ coincides with $(f_1,f_2)$ --- this would establish the fact that $\p$ is a projection on $\y$. We will, however, take a different route and obtain this result as a consequence of Proposition \ref{tdwa} presented in the next section (see Theorem \ref{ttrzy}, further down). 

\subsection{A complementary subspace to $\y$}\label{comp}

If $\y$ is complemented and $\p$ is a projection on $\y$, the range of the operator
\begin{equation} \q \coloneqq I_{\ce} -\p.\label{proj:7} \end{equation}
is a complement to $\y$. Our Proposition \ref{tdwa},  presented a bit further down,    
says that the range is contained in the subspace $\z$, composed of $(\gle,\gre) \in \ce$ such that 
\begin{align} 
\gle^e&=-\gre(0)\e_\su- \e_\su*(\be \gre -\al \gle^\es)  , \nonumber \\
\gre^e&=\gre(0)\e_\su+ \e_\su*(\be \gre -\al \gle^\es) . \label{comp:1a}
 \end{align}
%\begin{align}  \gle(x)&=-\gle(-x)-2\gre(0)\e^{-\su x}-2 \int_0^x\e^{-\su(x-y)}[\be \gre(y) -\al \gle(-y)] \ud y , \nonumber \\
%\gre(-x)&=-\gre(x)+2\gre(0)\e^{-\su x}+2 \int_0^x\e^{-\su(x-y)}[\be \gre(y) -\al \gle(-y)] \ud y. \label{comp:1a}  \end{align}
In other words,  the range of $\q$ is contained in
\[ \z \coloneqq \{ (g_1,g_2) \in \ce:\eqref{comp:1a} \text{ holds }\}.\]
This result, in turn, when combined with  Proposition \ref{raz}, allows showing that the range of $\q$ indeed forms a complement to $\y$ and that $\q$ and $\p$ are projections on these subspaces (see Theorem \ref{ttrzy}).

Before continuing, for ease of reference, we note that for $(\fle,\fre) \in \ce$,  the pair $(g_1,g_2) \coloneqq \q (\fle,\fre)$ is defined by (see \eqref{proj:gr})
\begin{align} 
\gle &= \fle^o - c \big (\textstyle{\frac{2\al}{\sq}}\sinh_\sq - \cosh_\sq  \big) +\big(\textstyle{\frac{2\al}{\gamma}} k_1-\textstyle{\frac{\gamma}{2}}  k_2\big ) * \sinh_\sq   -( k_1-\al  k_2) *  \cosh_\sq,\nonumber \\    
\gre & = \fre^o  - c \big (\textstyle{\frac{2\be}{\sq}}\sinh_\sq + \cosh_\sq \big )+\big(\textstyle{\frac{2\be}{\gamma}} k_1+\textstyle{\frac{\gamma}{2}} k_2\big) *\sinh_\sq  + ( k_1+\be  k_2)*\cosh _\sq,  \label{proj:g}
 \end{align}
where functions $ k_1,  k_2$  and the constant $c$ are given by, resp.,  \eqref{proj:F} and  \eqref{proj:C2}.

\begin{prop}\label{tdwa}
Let $(\fle,\fre)\in \ce$. Then $(g_1,g_2)$ defined by \eqref{proj:g}  belongs to $\z$. 
  \end{prop}

\begin{proof} 
The fact that $(g_1,g_2)= (\fle,\fre)- \p (\fle,\fre) $ belongs to $ \ce$ follows from Proposition \ref{raz} (it is only here that the value of $c$ is of importance). Hence, we need to show that \eqref{comp:1a} holds.
To this end, using \eqref{proj:g}, we find that
 \begin{align*}
  g_1^e & = c\cosh_\gamma -  \textstyle{\frac{\gamma}{2}}  k_2* {\sinh_\gamma} - k_1 *{\cosh_\gamma}=- g_2^e  .
 \end{align*}
Thus, the second relation of \eqref{comp:1a} is satisfied whenever the first one is.

On the other hand, by  \eqref{proj:g},  \eqref{intro:T} and \eqref{intro:T1}, we have
\begin{align*} \be g_2 - \al {g_1^\es} &=   k_1 - c(\gamma {\sinh_\gamma} +(\alpha+\beta) {\cosh_\gamma}) 
\\
&\phantom{=}+ \gamma \big( k_1+ \pol  (\su)  k_2\big)* {\sinh_\gamma}+\big((\su)  k_1+ \textstyle{\frac{ \gamma^2}{2}}  k_2\big)* {\cosh_\gamma}.\end{align*}
Since, by \eqref{proj:g}, $g_2(0)=-c$, the second relation in \eqref{proj:dod} reduces our task to proving that 
\begin{align*}
{ k_1}* [ {e_{\alpha+\beta}} * (  \gamma & {\sinh_\gamma} + (\alpha+\beta) {\cosh_\gamma})-{\cosh_\gamma}+e_{\alpha+\beta}  ]
\\& = \textstyle{\frac  \gamma 2}{ k_2}* [ {\sinh_\gamma} - {e_{\alpha+\beta}}*((\su)  {\sinh_\gamma} + \gamma  {\cosh_\gamma})]. \end{align*}
To complete the proof we observe that, by \eqref{proj:dod}, both expressions in brackets are identically zero.
\end{proof}

\begin{thm}\label{ttrzy}The space $\ce $ is a direct sum of two subspaces:
\[ \ce = \y \oplus \z .\] 
Moreover, $\p$ is a projection on $\y$ and $\q$ is a projection on $\z$. \end{thm}
\begin{proof}
Since, by Proposition \ref{raz} and Proposition \ref{tdwa}, \[ \ce = \y + \z,\] to prove the first part  we need to show only that  $\y \cap \z = \{0\}.$

To this end, we take $(\gle,\gre)\in \y \cap \z$.
By $\gle^\es = \gle^e - \gle^o$ and $\gre = \gre^e + \gre^o$,  relations     \eqref{proj:1aa}  and \eqref{comp:1a} yield
\begin{align}
\gle^\es &=-\gre(0)e_\su - e_{\alpha+\beta} *[(\alpha+\beta)\gre-2\al\gle^\es],  \nonumber  \\ 
\gre&=\gre(0)e_\su  +e_{\alpha+\beta}* [2\be\gre-(\alpha+\beta)\gle^\es] . \label{sum1}
\end{align}

We note that all functions that feature here are of at most exponential growth. Therefore, their Laplace transforms are well-defined, and in the Laplace transform terms this system takes the form:
\begin{align*}
(\lam -\al+\be) \wh {\gle^\es} (\lam) + (\alpha+\be) \wh \gre (\lam) &=- g_2(0),\\
(\al+\be) \wh {\gle^\es }(\lam) + (\lam +\al-\be) \wh \gre(\lam) &= g_2(0).\end{align*}
 As long as  $\lam >\gamma$, this system has a unique solution,  given by 
\begin{align*}
\wh {\gle^\es}(\lam) = -g_2(0) \frac{\lam +2 \al }{\lam^2 -\gamma^2}, \;\;\;\;\;\;
\wh \gre(\lam) = g_2(0) \frac{\lam  +2\beta}{\lam^2 -\gamma^2} . 
\end{align*}
Therefore, $
\gre(x)=g_2(0)( \cosh_\gamma(x)+2\be \gamma^{-1}\sinh_\gamma (x))$ for $x\ge 0.$ 
 Since $g_2 \in \cer $, however, we have to have $g_2(0)=0$. Thus $(g_2)_{[0,\infty)}=(g_1^T)_{[0,\infty)}= 0$. Now, 
 Lemma \ref{lem:ab2} says that for $(g_1,g_2)\in \y$, the parts $(g_1^T)_{[0,\infty)}$ and $(g_2)_{[0,\infty)}$ determine the entire pair (comp. the end of the proof of Proposition \ref{propek1}) In particular, $(g_2)_{[0,\infty)}=(g_1^T)_{[0,\infty)}= 0$ implies $g_1=g_2=0$, completing the proof of the first part. 
 
The second part now easily follows. To wit, if $(f_1,f_2)$ belongs to $\y$, then so does $(f_1,f_2) - \p (f_1,f_2)$, because of Proposition \ref{raz}. On the other hand, by Proposition \ref{tdwa}, $(f_1,f_2) - \p (f_1,f_2)$ belongs to $\z$. 
In the first part we have proved, however, that $\y \cap \z=\{0\}$. Hence, $\p (f_1,f_2)=(f_1,f_2).$ Similarly, we show that $\q (f_1,f_2) = (f_1,f_2)$ for $(f_1,f_2) \in \z.$ This together with Propositions \ref{raz} and \ref{tdwa} completes the proof. \end{proof}

\section{ $\z$ as an invariant subspace related to conditions \eqref{intro:bc2}}\label{gener}
So far, our main result (Theorem \ref{ttrzy}) establishes that $\mc C$ is a direct product of two subspaces, one of which, namely $\y$, is invariant under the basic Cartesian product cosine family (and thus for the Cartesian product of two copies of Brownian motion semigroup, as well). Moreover, by Theorem \ref{ab:genth}, the snapping out Brownian motion semigroup and the related cosine family are in fact isomorphic (similar) to the subspace semigroup and the subspace cosine family in $\y$. 
At present, however, it is yet unclear whether $\z$ has a similar property. The aim of this section is to show that $\z$ is indeed an invariant subspace and that the related subspace semigroup and cosine family are isomorphic (similar) to the semigroup and cosine family generated by a Laplace operator with transmission conditions \eqref{intro:bc2}. Invariance of $\z$ is established in Theorem \ref{thm:inv}; connection with transmission conditions is provided in Theorem \ref{gen_th}.

\subsection{Further invariant subspaces for the Cartesian product cosine family}

We take a similar approach to that presented in Section \ref{exten}, that is, we start with some observations  on the basic cosine family viewed as composed of linear operators acting in $\mathfrak C(\mathbb R)$. We note, namely, that for any $a>0$,
\begin{equation*} C(t) e_a (x) = \e^{-ax}e_a^e (t) \ \text{ and } \   C(t)\phi (0)=\phi^e (t),  \qquad t,x\in \R.  \end{equation*}
 Therefore, by \eqref{ab:1}, we see that for all $t,x\in \R$ and $\phi \in \mathfrak C(\R)$:
\begin{align}\label{gener:2} C(t) (ae_a * \phi+\phi(0)e_a) (x) &=(a e_a * C(t)\phi +   [C(t)\phi (0)]e_a)(x) \\
 &\phantom{=}-\e^{-ax} [\phi  - ae_a*\phi -\phi(0)e_a  ]^e (t).  \nonumber \end{align}
This counterpart of \eqref{ab:1} shows that also functions $\phi \in \mathfrak C(\R)$ such that $\phi  -  a(e_a*\phi) -  \phi(0)e_a$ is odd,  are of special importance for the basic cosine family.  Here is a lemma that summarizes their basic properties.% of such functions.   

\begin{lem}\label{lem:gener1}For $\phi \in \mathfrak C(\R)$ the following conditions are equivalent. 
\begin{enumerate} 
\item $\phi^e=  a e_a *\phi+\phi(0)e_a$ on $\R^+$;
\item $\phi^e= a e_a *\phi+\phi(0)e_a$ on $\R$;
\item $\phi  - ae_a*\phi -\phi(0)e_a$ is odd, %$-a e_a^\es * \phi^\es+\phi(0)e_a^\es =a e_a * \phi +\phi(0)e_a$ on $\R$;
\item For all $t\in \R$, %we have
 \begin{equation} C(t)(a  e_a * \phi+\phi(0)e_a)=a e_a *C(t)\phi+[C(t)\phi(0)]e_a \quad \text{on } \R. \end{equation}
\end{enumerate}
 \end{lem}
\begin{proof}  We proceed as in Lemma \ref{lem:ab1}.
In order to show (b), if (a) is assumed, we need to check, by \eqref{intro:T1}, that $\phi^e=- a e_a^\es *\phi^\es+ \phi(0)e_a^\es$ on $\R^+$, that is, by (a), that \begin{align}-a e_a^\es * \phi^\es+\phi(0)e_a^\es=a e_a * \phi +\phi(0)e_a\qquad \text{ on } \R^+.\label{gener:3} \end{align} 
On the other hand, using assumption (a) again, we obtain that the left hand side of \eqref{gener:3} equals
\begin{align*}-a  e_{-a} *(-\phi+ &2ae_a*\phi+2\phi(0)e_a)+\phi(0)e_{-a}  \\
&=a(e_{-a} - 2a e_{-a}* e_a) * \phi+\phi(0)(e_{-a} - 2a e_{-a}* e_a) , \end{align*}
and the last expression is equal to the right hand side of \eqref{gener:3} by the Hilbert equation \eqref{hilbert}. This completes the proof of the implication. The converse is trivial. 

Next, (b) says that $\phi  - ae_a*\phi -\phi(0)e_a$ is the odd part of $\phi$, and thus clearly implies (c). 
Conversely, (c) says that 
\begin{equation} \phi^e=a(e_a*\phi )^e+\phi(0)\cosh_a. \label{fis:1} \end{equation}
When expanded, this reads $\phi - a (e_a * \phi) + \phi^\es + a (e_{-a} *\phi^\es) = 2\phi (0) \cosh_a $. 
Since the left-hand side here is the derivative of $e_a *\phi + e_{-a}*\phi^\es $, we see that the derivatives of $(e_a *\phi)^o$ and $a^{-1} \phi (0)\sinh_a$ coincide. It follows that $a(e_a *\phi)^o = \phi (0)\sinh_a$ because both functions vanish at $0$. This, together with \eqref{fis:1} renders
$a e_a * \phi = a (e_a * \phi)^e + a (e_a * \phi)^o = \phi^e - \phi(0)e_a $, and thus proves (b).  
Finally, (c) is equivalent to (d) by \eqref{gener:2}, completing the proof. 
 \end{proof}

As a direct consequence of the lemma we obtain the following information on invariant subspaces for $\{\ced (t),\, t\in \R\}$.  
\begin{cor}\label{invariant} The subspaces 
\begin{align*}\mathsf  D_{a} & \coloneqq \{(\phi_1,\phi_2) \in \ce:{\textstyle \frac  a 2}\phi_1^e = ae_a*\phi_2+\phi_2(0)e_a= \phi_2^e\},
\\
\mathsf D_a^\sharp & \coloneqq \{(\phi_1,\phi_2)\in \ce: \phi_1^e = a e_a * \phi_1+\phi_1(0)e_a, \phi_2^e=0\}
\end{align*}
 are invariant under the basic Cartesian product cosine family. \end{cor}
\begin{proof} To show invariance of $\mathsf D_{a}$, similarly as in Corollary \ref{cor:ab1}, we need to prove that
$(C(t)\phi_2)^e = ae_a*C(t)\phi_2+[C(t)\phi_2(0)]e_a$ and  $({\textstyle \frac  a 2}C(t)\phi_1)^e =( C(t)\phi_2)^e$ 
for $(\phi_1,\phi_2) \in  \mathsf D_a$ and $t\in \R$. The second equality, however, is an immediate consequence of the second relation in \eqref{tran} combined with the defining condition of $\mathsf D_a$. Turning to the first equality, we  note that $\phi_2$  satisfies condition (b) in Lemma \ref{lem:gener1}, and thus  satisfies also the lemma's condition (d). Therefore, using \eqref{tran} and the defining condition of $ \mathsf D_a$  again, we obtain
\[(C(t)\phi_2)^e = C(t)\phi_2^e=C(t)(a  e_a * \phi_2+\phi_2(0)e_a) 
=a e_a *C(t)\phi_2+[C(t)\phi_2(0)]e_a,\]
as desired. Proof of invariance of $ \mathsf D_{a}^\sharp$ is analogous. \end{proof}

Our next result establishes an intimate, key connection between the space $\z$ defined in Section \ref{comp} and the spaces $\mathsf D_\su$ and $\mathsf D_{\alpha+\beta}^\sharp$. 

\begin{lem} \label{lem:gener2}
We have  \[ \z=M(\mathsf D_\su), \] where $M$ is defined by \eqref{ab:2.5}, as long as $\al \neq \be$. In the other case, 
\[ \mathcal D_{\alpha,\alpha} = M^\sharp (\mathsf D_{\alpha+\beta}^\sharp),\]
 for $M^\sharp$ introduced in \eqref{ab:4.5}.
\end{lem}

\begin{proof} We argue similarly as in Lemma \ref{lem:ab2}. Let us assume that $\alpha \not =\beta$, a pair $(\phi_1,\phi_2)$ belongs to $\mathsf D_{\alpha+\beta}$, and $(\psi_1,\psi_2)$ satisfies \eqref{added}. Then, by the definition of $\mathsf D_{\alpha+\beta}$, 
\begin{align*} \binom{\psi_1^e}{\psi_2^e} &=   \frac1{\alpha-\beta} \binom{\beta \phi_1^e-\phi_2^e}{\alpha \phi_1^e-\phi_2^e} =\binom{ -e_{\alpha+\beta} *\phi_2-(\alpha+\beta)^{-1} \phi_2(0) e_{\alpha+\beta}}{
 e_{\alpha+\beta} *\phi_2+(\alpha+\beta)^{-1} \phi_2(0) e_{\alpha+\beta}}.
 \end{align*}
Also, by the same definition, $\frac {\alpha+\beta} 2 \phi _1(0) = \phi_2(0)$ and this implies that $\psi_2(0)=(\alpha+\beta)^{-1} \phi_2(0)$. It follows that the lower entry in the last matrix is $ e_{\alpha+\beta} * (\beta \psi_2 - \alpha\psi_1^\es)+\psi_2(0) e_{\alpha+\beta}$. This shows that $(\psi_1,\psi_2)$ belongs to $\z$ (see \eqref{comp:1a}), that is, that $\z \supset M(\mathsf D_{\alpha+\beta})$. The inclusion  $\z \subset M(\mathsf D_{\alpha+\beta})$ is proved similarly.

We omit the details of the case $\alpha=\beta$. 
\end{proof}

Since, as we have already remarked, operators $M$ and $M^\sharp$  commute with the basic Cartesian product cosine family, Corollary \ref{invariant} and Lemma \ref{lem:gener2}, when combined, yield the following crucial result.   

\begin{thm}\label{thm:inv} The space $\z$ is invariant under the basic Cartesian product cosine family. \end{thm}

We also note the following corollary to Lemma \ref{lem:gener2}, that will turn out important in Section \ref{cwtc}. 

\begin{cor}\label{gen:cor}
For $(\psi_1,\psi_2)\in \ce$ the following conditions are equivalent. 
\begin{enumerate} 
\item $(\psi_1,\psi_2)$ belongs to $\z$, 
\item $\psi_2^e=  e_{\alpha+\beta} *(\beta\psi_2 - \alpha \psi_1^\es)+\psi_2(0)e_\su=-\psi_1^e$
on $\R_+$.
\end{enumerate}

\end{cor}

\begin{proof} From  the characterization \eqref{comp:1a} it follows that  all we need to prove is that (b) implies (a). Let $\alpha \not =\beta$. Assuming (b), we see that for $\binom{\phi_1}{\phi_2}\coloneqq M^{-1} \binom {\psi_1}{\psi_2}= \binom{\psi_2- \psi_1^\es}{\beta \psi_2-\alpha \psi_1^\es}$ 
we have \begin{equation}\label{lojej} \pol  (\alpha+\beta) \phi_1^e = (\alpha+\beta) e_{\alpha+\beta}*\phi_2+\phi_2(0)e_{\alpha+\beta}= \phi_2^e \end{equation}
on $\R^+$. Since this means that $\phi_2$ satisfies condition (a) in Lemma \ref{lem:gener1} (with $a=\alpha+\beta$), the lemma implies that 
$(\alpha+\beta) e_{\alpha+\beta}*\phi_2+\phi_2(0)e_{\alpha+\beta}= \phi_2^e $ holds on $\R$, and thus that \eqref{lojej} holds on $\R$ also. By Lemma \ref{lem:gener2} this shows that $\binom{\psi_1}{\psi_2}$ belongs to $\z$. 

The case of $\alpha=\beta$ is proved similarly. \end{proof}

\subsection{Connection with transmission conditions \eqref{intro:bc2}}\label{cwtc}
We define the space $\xo$ as the subspace of $\x$ composed  of functions satisfying $f(0+)=-f(0-)$ (`$\textnormal{ov}$' stands for `opposite values').
\begin{prop}\label{lem:5.3}
 Let the map $ E^\perp_{\al,\be}  \colon \xo\to \ce$ be given by $f\overset{ E^\perp _{\al,\be}}\longmapsto (\wfl,\wfr)$, where
\begin{align}
\wfl (-x) &= f(-x),  \qquad \wfr (x) = f(x),
 \nonumber \\
\wfl (x) &= -f^\es(x)-2f(0+)e_\su(x)-2 e_\su * [\be f-\al f^\es](x), \nonumber %\label{proj:0} 
\\
\wfr (-x)&=-f(x)+2f(0+)e_\su(x)+ 2 e_\su * [\be f-\al f^\es](x), \quad  x> 0, \label{gener:ex} \end{align}
and $\wfl(0)=f(0-)$, $\wfr(0)=f(0+)$. 
Then $ E^\perp_{\al,\be} $ is an isomorphism of $\xo$ and $\z$.\end{prop}

\begin{proof} 
To conclude that $\wfl,\wfr\in \cer$ we use  Lemma $\ref{lopital}$ and the existence of  $\wfl(-\infty)$ and $\wfr(\infty)$, which is guaranteed by $f\in \x$. The definition of $ E^\perp_{\al,\be} $ implies also that $(\wfl,\wfr)$ satisfies condition (b) in Corollary \ref{gen:cor}
 and thus belongs to  $\z$. 

Moreover, $ E^\perp_{\al,\be} $ is an injective bounded linear operator with 
\begin{equation}\|E^\perp_{\al,\be}\| \le 5.\label{piatka}\end{equation}
It is also surjective, since the operator $ R$ of Section \ref{tro}, as restricted to $\z$, %is  ^\perp \colon \z \to \xo$ given by $ R (\psi_1,\psi_2)=f$, where $f(x)=\psi_1(x), x<0$ and $f(x)=\psi_2(x), x>0$, 
is its right inverse. %and $ R $ is bounded with norm equal to $1$.
\end{proof}

To establish the link between the subspace $\z$ and boundary conditions \eqref{intro:bc2} we introduce the corresponding generator:

\begin{defn} Let $A^\perp_{\al,\be}$ be the operator in  $\xo$, defined by \[A^\perp_{\al,\be} f = f''\] on the domain $D(A^\perp_{\al,\be})$ consisting of functions $f \in \xo$ satisfying the following
conditions: 
\begin{itemize}
\item[(a)] $f$ is twice continuously differentiable in both
  $(-\infty,0]$ and $[0,\infty)$, separately, with left-hand and
  right-hand derivatives at $x=0$, respectively, 
\item[(b)] both the limits $\lim_{x\to \infty} f'' (x)$ and
  $\lim_{x\to -\infty} f''(x) $ exist and are finite,
 \item[(c)] $f''(0+)=\alpha f'(0-) + \beta f'(0+)$ and $f''(0+)=-f''(0-).$   \end{itemize}
\end{defn}

\begin{thm}\label{gen_th}
The operator $A^\perp_{\al,\be}$ generates the cosine family   $\{\cosabt(t),\,t \in \R\}$ given by
\begin{equation} \label{gener:B} \cosabt (t)  \coloneqq   R C_D(t) E^\perp_{\al,\be}, \end{equation}
where $C_D(t)$ is defined   in  \eqref{intro:7}. \end{thm}

\begin{proof} 

This result may be proved in much the same way as Theorem \ref{ab:genth}.
%We proceed analogously to the proof of Theorem \ref{ab:genth}. 
The crucial step of the argument is to show that, for $f\in \x$, the
condition $f\in D(A^\perp_{\al,\be})$ holds iff extensions $\wfl, \wfr$ defined in Proposition \ref{lem:5.3} are 
twice continuously differentiable with  $\wfl''$, $\wfr''$ belonging to $\cer$.  Transmission condition of point  (c) above is a necessary and sufficient condition for differentiability of extensions at $x=0$. \end{proof}

\section{Approximation of skew Brownian Motion: convergence of solution families}\label{aosb} 

\subsection{A convergence theorem}
In this section, we give additional insight into the theorem obtained in \cite{tombatty} and repeated in \cite{knigaz}*{Chapter 11}, saying that, as the permeability coefficients converge to infinity in such a way that their ratio remains constant, the solution families of snapping-out Brownian motions converge to those of the skew Brownian motion.

The convergence result alluded to above says that each $\aab$ is a Feller generator (i.e., the generator of a strongly continuous family of positive contraction operators in $\x$) and that  
formula \eqref{intro:cosik} holds. To repeat, this formula says that for $s>0$, 
\begin{equation}\label{skew:dodane} \gran \sup_{t\in [0,s]} \|\e^{t \anab}f - \e^{t\askew}f\| = 0 \end{equation}
for all $f \in \cer \subset \x$. Here, 
$\askew $ is the following Feller generator in $\cer$.  An $f\in \cer$ belongs to its domain
if the following three
conditions are satisfied: 
\begin{itemize}
\item[(a)] $f$ is twice continuously differentiable in both
  $(-\infty,0]$ and $[0,\infty)$, separately, with left-hand and
  right-hand derivatives at $x=0$, respectively, 
\item[(b)] both the limits $\lim_{x\to \infty} f'' (x)$ and
  $\lim_{x\to -\infty} f''(x) $ exist and are finite {(it follows that, in fact, they
  have to be equal to $0$)}, and
 \item[(c)] $\alpha f'(0+) = \beta f'(0-)$ and $f''(0+)=f''(0-).$  Note that this condition implies that, although $f'(0)$ need not exist, it is meaningful to speak of $f''(0)$. \end{itemize}
 Furthermore, 
 \[ \askew f = f''.\]

Turning to the particulars of   \eqref{skew:dodane}, we recall that the main idea of the Trotter--Kato--Neveu convergence theorem \cite{abhn,engel,goldstein,pazy}, a cornerstone of the theory of convergence of semigroups \cite{knigaz,bobrud}, is that convergence of resolvents of equibounded semigroups in a Banach space $\mathsf X$, gives an insight into convergence of the semigroups themselves. Hence, in studying the limit of semigroups, say, $\{\e^{tB_n}, t \ge 0\}, n\ge 1$, generated by the operators $B_n, n \ge 1$ we should first establish existence of 
the strong limit 
\[ \rla \coloneqq \gran \rez{B_n}. \]
The general theory of convergence (see \cite{knigaz}*{Chapter 8}) covers also the case in which, unlike in the  classical version of the Trotter--Kato--Neveu theorem, the (common) range of the so-obtained operators $\rla, \lam >0$ is \emph{not} dense in  $\x$, and stresses the role of  the so-called regularity space, defined as the closure of the range of $\rla$: 
\begin{equation}\label{skew:reg} \xreg \coloneqq cl (Range \rla) \subset \mathsf X. \end{equation}
Namely, $\xreg$ turns out to be composed of $f\in \x$ such the limit $T(t)f \coloneqq \gran \e^{tB_n}f $ exists and is uniform with respect to $t$ in compact subintervals of $[0,\infty)$; then $\{T(t),\, t \ge 0\}$ is a strongly continuous semigroup in $\xreg$.

Condition \eqref{skew:dodane} is thus a  typical result of convergence theory: it characterizes the regularity space for $B_n\coloneqq \anab$ as equal to $\cer$, and identifies $\{T(t),\, t\ge 0\}$ as $\{\e^{\askew}, t \ge 0\}$.

It should be stressed, though, that this result does not exclude the possibility of existence of $f\not \in \cer$ such that $\jcg{\e^{t\anab}f}$ converges for all $t>0$. Such \emph{irregular} convergence of semigroups, which is known to be always uniform with respect to $t$ in compact subsets of $(0,\infty)$ --- see \cite{note} or \cite{knigaz}*{Thm 28.4} --- is not so uncommon, especially in the context of singular perturbations \cite{banmika,banalachokniga,knigaz}. We will prove that the limit $\gran \e^{t\anab}f$ exists for all $t>0$ and $f\in \x$.

Moreover, \eqref{skew:dodane} says nothing about convergence of the related cosine families. In fact, the analysis presented in \cite{tombatty} does not even guarantee that these cosine families are equibounded (see Remark 6.1 in that paper). We are able to improve \eqref{skew:dodane} as follows. 

\begin{thm}\label{thm:skew} \ \ 
\begin{itemize}
\item [(a)] The cosine families $\{ \cosab (t), \,t \in \R\}$ generated by $\aab$ are equibounded: 
\begin{equation} \|\cosab (t)\| \le 5 \qquad \alpha,\beta\ge 0, t \in \R.\label{skew:2}\end{equation}
\item [(b)] We have 
\begin{equation}\label{skew:3} \gran \sup_{t \ge 0} \| \cosabn (t)f - \cskew (t) f\| =0 \qquad  f \in \cer, \end{equation}
 where $\{\cskew (t), \, t \in \R\}$ is the cosine family generated by $\askew$.  If $f$ does not belong to $\cer$, this limit does not exist for at least one  $t\in \R \setminus \{0\}$. \newline 
\item [(c)] The limit
\begin{equation}\label{skew:4} \gran \e^{t\anab} f \end{equation}
exists for all $t>0$ and $f\in \x,$ and for $f\not \in \cer $ it is uniform with respect to $t $ in compact subintervals of $(0,\infty)$. \end{itemize} 
\end{thm}

\begin{rem}\label{rem:one} Thesis (b) and the Weierstrass formula (see e.g. \cite{abhn}*{p. 219}) imply 
that the limit in  \eqref{skew:dodane} is in fact uniform in $t\ge 0$. To see this, let $f\in \cer $ and $\eps >0$ be fixed and let $n_0$ be so large that 
\[ \| \cosabn (s) f - \cskew  (s)f \|\le  \eps \]
for $n\ge n_0$ and $s\in \R$. Then, for any $t>0$,
\begin{align*} \|\e^{t\anab} f - \e^{t\askew} f\|&\le  {\textstyle \frac 1{2\sqrt{\pi t}}} \int_{-\infty}^\infty \e^{-\frac {s^2}{4t}} \|\cosabn (s) f  -  \cskew f \| \ud s \\
 &=  {\textstyle \frac 1{2\sqrt{\pi t}}} \int_{-\infty}^\infty \e^{-\frac {s^2}{4t}} \eps \ud s =\eps, 
 \end{align*} 
(and for $t=0$ the left hand side is $0$). In particular, (b) improves \eqref{skew:dodane}.% obtained in \cite{tombatty}. 
\end{rem}

Interestingly, the approximation theorem  presented above is accompanied by its counterpart devoted to complementary transmission conditions, that is, to cosine families $\{\cosabt,\, t\in \R\}$ described in Theorem \ref{gen_th}. As it turns out, these families converge on the entire $\xo$, and, yet more surprisingly, the limit cosine family is a mirror image of that related to the skew Brownian motion. A precise statement is  contained in the following theorem.

\begin{thm}\label{thm:weks} Let $ J$ be the isometric isomorphism of  $\xo$ and $\cer$ given by $J f (x)= -f(x), x<0$, $Jf(x)=f(x), x >0$ and $Jf(0)= f(0+).$ Then, 
\begin{itemize}
\item [(a)] The cosine families $\{ \cosabt (t), \,t \in \R\}$ generated by $A_{\al,\be}^\perp$ are equibounded: 
\begin{equation} \|\cosabt (t)\| \le 5 \qquad \alpha,\beta\ge 0, t \in \R.\label{weks:2}\end{equation}
\item [(b)] We have 
\begin{equation}\label{weks:3} \gran \sup_{t \in \R} \| \cosabnt (t)f - J^{-1} \cskewt (t)J f\| =0 \qquad  f \in \xo, \end{equation}
\item [(c)] Finally, 
\begin{equation*} \gran \sup_{t \ge 0} \| \e^{tA_{n\al,n\be}^\perp}  - J^{-1} \e^{tA_{\be,\al}^\perp} J f\| =0 \qquad  f \in \xo. \end{equation*}
\end{itemize} 
\end{thm}

\begin{rem}\label{weks:gen}
The limit cosine family  $\{J^{-1} \cskewt (t)J, \, t\in \R\}$ is generated by the isomorphic image of the generator $A_{\be,\al}^{\textnormal{skew}}$ in $\xo$, that is, the operator  
$\aweks$ defined by $\aweks f = f'',$
on the domain composed of $f\in \xo$  satisfying the following conditions: 
\begin{itemize}
\item[(a)] $f$ is twice continuously differentiable in both
  $(-\infty,0)$ and $(0,\infty)$, separately, with left-hand and
  right-hand derivatives at $x=0$, respectively, 
\item[(b)] both the limits $\lim_{x\to \infty} f'' (x)$ and
  $\lim_{x\to -\infty} f''(x) $ exist and are finite,
 \item[(c)] $\be f'(0+) = -\al f'(0-)$ and $f''(0+)=-f''(0-).$   \end{itemize}
\end{rem}

%\subsection{Proof of \eqref{skew:2}}

\subsection{Proof of Theorem \ref{thm:skew} (a) and (b)}
By definition \eqref{ab:6}, estimate \eqref{skew:2} is a direct consequence of \eqref{ab:5a} because the operator norms of $R$ and $\ced(t)$ are equal to $1.$ 

Turning to the proof of (b), we note that in the defining formula \eqref{ab:6}, $E_{\al,\be}$ (introduced in Proposition \ref{propek1}) is the only operator that depends on $\alpha$ or $\beta$.  More importantly, to prove existence of the limit $\gran \cosabn (t) f $  for $f\in \cer$, it suffices to prove existence of $\gran \enab f$. As we will prove now, however, the latter is an immediate consequence of the fact that distributions of exponential random variables with large parameters converge to Dirac measure at $0$, as expressed in Lemma \ref{lem:dirnew} presented in Appendix.

\begin{prop}\label{skew:ext_prop}
We have
\begin{equation}\label{skew:5} \gran \enab f = E_{\alpha,\beta}^{\textnormal{skew}} f, \qquad f \in \cer \end{equation}
where $E_{\alpha,\beta}^\textnormal{skew} f = (\widetilde {f_{\ell}},\widetilde {f_{\textnormal r}})$ is defined by
\begin{align}\label{skew:4.1}
 \widetilde {f_\ell}(x)&= \begin{cases*}f(x),& {for } $x\leq  0$, \\ 
\frac{\be-\al}{\apb} f(-x)+\frac{2\al}{\apb} f(x), & {for } $x> 0,$  \end{cases*}
 \end{align}
and
\begin{align}\label{skew:4.2}
  \widetilde {f_{\textnormal r}}(x)&= \begin{cases*} \frac{2\be}{\apb} f(x)+\frac{\al-\be}{\apb} f(-x) , & {for } $x< 0$,\\
f(x),& {for } $x\geq  0$.  \end{cases*} \end{align}  
\end{prop}
\begin{proof}Lemma \ref{lem:dirnew} (c)  implies that for $f\in \cer$, the integral found in the second line of formula \eqref{proj:1} with  $\alpha$ and $\beta$ replaced by $n\alpha$ and $n\beta$, respectively,  converges, as $n\to \infty$, to $(\alpha+\beta)^{-1} (f(x) - f(-x))$ uniformly in $x\ge 0.$ This is because, for $f\in \cer$,  the function $\phi$ defined by $\phi (x) = f(x) - f(-x), x\ge 0$ belongs to $\cerp$, and we have $\phi (0)=0$.

 It follows that the first coordinate of $ \enab f $  converges (in the norm of $\cer$) to $\widetilde {f_{\ell}}$ defined by  \eqref{skew:4.1}.
Similarly, its second coordinate converges to $ \widetilde {f_{\textnormal r}}$ from \eqref{skew:4.2}.
 \end{proof}

There are now at least two natural ways to prove \eqref{skew:3}. The first one, a bit more direct, is to note that we have just established  that 
\[ \gran \cosabn (t) f = R\ced (t)  E_{\al,\be}^\textnormal{skew} f, \qquad t \in \R, f \in \cer . \]
On the other hand, in \cite{abtk}*{Section 6}, Lord Kelvin's method of images with the same extension operator $ E_{\alpha,\beta}^\textnormal{skew} $ has been used to prove a generation theorem for $\askew$. In other words, it was proved there that 
\begin{equation}\label{abskew} \cskew (t) = R\ced (t)  E_{\al,\be}^\textnormal{skew}, \qquad t \in \R. \end{equation}  
This immediately renders \eqref{skew:3}. 

To present the second method, we introduce 
\[\cosik (t) f \coloneqq \gran \cosabn (t)f \]
and note that, as a result of \eqref{skew:5}, the above limit is uniform with respect to $t$ in the entire $\R$.  
Hence, $\{\cosik (t),\, t \in \R\}$ is a strongly continuous family of operators, and it is a cosine family because so are $\{\cosabn (t),\, t \in \R\}, n \ge 1.$ Moreover, by \eqref{skew:3}, $\|\cosik (t) \|\le 5.$

Let $G$ be the generator of $\{\cosik (t),\, t \in \R\}$. Then, by the Lebesgue Dominated Convergence Theorem, for $\lam >0$ and $f \in \cer $, 
\begin{align*} 
\lam (\lam^2 - G)^{-1} f &= \int_0^\infty \e^{-\lam t} \cosik (t)  f\ud t = \gran  \int_0^\infty \e^{-\lam t} \cosabn (t)  f\ud t \\& = \gran \lam (\lam^2 - \anab )^{-1} f = \lam (\lam^2-\askew)^{-1} f, \end{align*}
with the last equality following by \eqref{skew:dodane}. Hence, we conclude that the resolvents of operator cosine functions generators $G$ and $\askew$ coincide, and this implies that $G=\askew$, proving \eqref{skew:3}. 

The rest of Theorem  \ref{thm:skew} (b) is a particular case of the general theorem proved in \cite{zwojtkiem} (see also \cite{knigaz}*{Chapter 61}) saying that outside of the regularity space (i.e., outside of the subspace defined in \eqref{skew:reg} --- the regularity space of a sequence of cosine families is by definition the regularity space of the sequence of corresponding semigroups) cosine families cannot converge. %In other words, irregular convergence of cosine families is impossible. 
Since in the case of cosine families $\{\cosabn (t),\, t \in \R\}, n \ge 1$ the regularity space equals $\cer$, 
outside of $\cer$ there is no $f$ such that the limit $\gran \cosabn (t)f $ exists for all $t\in \R$.   

\subsection{Proof of Theorem  \ref{thm:skew} (c)}
%\subsection{Proof of \eqref{skew:4}}
We start by recalling that, as an application of his algebraic version of the Hille--Yosida Theorem, J. Kisy\'nski has proved the following result accompanying Trotter--Kato--Neveu Theorem (see Corollary 5.8 in \cite{prof10}): if $B_n, n \ge 1$ are generators of equibounded semigroups, then existence of the strong limit 
\[ \gran \rez{B_n}  \]
is equivalent to existence of 
\[ \gran \int_0^\infty \Psi (t) \e^{tB_n} \ud t \]
for any $\Psi$ that is absolutely integrable on $\R^+$ (see also \cite{hn}; the fact that existence of the first limit implies existence of the second for $\Psi = 1_{[0,\tau]}$ with $\tau>0$, has been noted already in the 1970 paper of T. G. Kurtz \cite{kurtz2}). 

An analogue of this result for cosine families was found in \cite{odessa}.  In our case it says that existence of the strong limit $\gran \rez{\anab}$ (which is established in \cite{tombatty}) implies existence of  the strong limit 
\begin{equation}\label{skew:6} \gran \int_{-\infty}^\infty \Psi (s) \cosabn (s) \ud s \end{equation}
for any $\Psi$ that is absolutely integrable on the entire $\R$ and even (as long as we have estimate \eqref{skew:2}). 
Since, by the Weierstrass Formula (see above),
\[ \e^{t\anab} f =  {\textstyle \frac 1{2\sqrt{\pi t}}} \int_{-\infty}^\infty \e^{-\frac {s^2}{4t}} \cosabn (s) f \ud s, \qquad f \in \x,  t >0\] 
the limit in \eqref{skew:4} is a particular case of \eqref{skew:6} for $\Psi(s)= \Psi_t (s) \coloneqq   {\textstyle \frac 1{2\sqrt{\pi t}}} \e^{-\frac {s^2}{4t}},$ $ s\in \R, t >0$. The fact that the limit is uniform with respect to $t$ in compact subintervals of $(0,\infty)$ is a consequence of the general result discussed in \cite{note} and \cite{knigaz}*{Thm 28.4}. 

\subsection{Proof of Theorem \ref{thm:weks}}

The following lemma is a key to the proof of convergence of the families $\{\cosabnt, \,t\in\R\}$, as $n\to \infty$   (see Figure \ref{rys_lim}).

%\begin{center}
%\begin{figure} 
%\includegraphics[scale=0.44]{rysunki/limit.eps}
%\caption{Convergence of extensions of a single function $f \in \xo$ (black lines) for $\al=0.2,\be=0.1$:
% yellow, green and violet lines are $E^\perp_{n\al,n\be}f$ for $n=1,2,10,$ resp.,
% blue line is $E_{\alpha,\beta}^\textnormal{weks}f$.}
%\label{rys_lim}
%\end{figure}
%\end{center}

\begin{prop}\label{weks:prop}
We have
\begin{equation}\label{weks:5} \gran E^\perp_{n\al,n\be}  f = E_{\alpha,\beta}^{\textnormal{weks}} f, \qquad f \in \xo \end{equation}
where $E_{\alpha,\beta}^\textnormal{weks} f = (\widetilde {f_{\ell}},\widetilde {f_{\textnormal r}})$ is defined by
\begin{align}\label{weks:1.1}
 \widetilde {f_\ell}(x)&= \begin{cases*}f(x),& {for } $x\leq  0$, \\ 
\frac{\al-\be}{\apb} f(-x)-\frac{2\be}{\apb} f(x), & {for } $x> 0.$  \end{cases*}
 \end{align}
and
\begin{align}\label{weks:1.2}
  \widetilde {f_{\textnormal r}}(x)&= \begin{cases*} -\frac{2\al}{\apb} f(x)+\frac{\be-\al}{\apb} f(-x) , & {for } $x< 0$,\\
f(x),& {for } $x\geq  0$,  \end{cases*} \end{align}  
\end{prop}
\begin{proof} For $f\in \xo$,  the function $\phi$ defined by $\phi (x) = \be f(x) - \al f(-x), x> 0$, $\phi (0) =\be f(0+) -\al f(0-)=(\alpha+\beta)f(0+)$ belongs to $\cerp$. Therefore, by
Lemma \ref{lem:dirnew} (c), 
\[ f(0+)e_{n(\alpha+\beta)}(x)+n e_{n(\su)} * [\be f-\al f^\es](x) \]
 converges, as $n\to \infty$, to $(\alpha+\beta)^{-1}[\be f(x)-\al f^\es(x)]$ uniformly in $x\ge 0.$
 It follows that the first coordinate of $ E^\perp_{n\al,n\be} f $ ($E^\perp_{\alpha,\beta}$ being defined in Proposition \ref{lem:5.3}) converges in the norm of $\cer$ to $\widetilde {f_{\ell}}$ defined by  \eqref{weks:1.1}.
Similarly, its second coordinate converges to $ \widetilde {f_{\textnormal r}}$ from \eqref{weks:1.2} .
 \end{proof}

Turning to the proof of Theorem \ref{thm:weks}, we note first that its point (a) is a direct consequence of the estimate \eqref{piatka} combined with the fact that 
the norms of $R$ and $\ced(t)$ are equal to $1.$ 

As for point (b), we note that Proposition \ref{weks:prop} yields
\[ \gran \cosabnt (t) f = R \ced (t)  E_{\al,\be}^\textnormal{weks} f,\]
uniformly with respect to $t\in \R$ for all $f\in \xo$. Hence, we are left with identifying the cosine family $ \{R \ced (t)  E_{\al,\be}^\textnormal{weks},\, t \in \R\} $ with $\{J^{-1}\cskewt (t)J,\, t \in \R\}$, that is, with proving that 
\[ R \ced (t)  E_{\al,\be}^\textnormal{weks} = J^{-1}\cskewt (t)J, \qquad t \in \R. \]

To this end, we let $\mc J$ be the isometric automorphism of  $\ce $ given by $\mc J (f_1,f_2) = (-f_1,f_2)$. It is then clear that  $\mathcal{J}$ is its own inverse. Moreover, in view of \eqref{abskew}, to complete the proof we need to establish the following three identities: 
\begin{itemize} 
\item [(i) ]$ J^{-1}R\mc J (f_1,f_2) =R (f_1,f_2),$ provided that $f_1(0)=-f_2(0)$. 
\item [(ii) ] $\mc J\ced (t) \mc J = \ced (t),t\in \R$,
\item [(iii) ] $\mc J E_{\be,\al}^\textnormal{skew} J = E_{\al,\be}^\textnormal{weks}$. % \tcm{po lewej: skew}
\end{itemize}
This, however, can be achieved by a straightforward calculation. 

Since point (c) can be obtained from (b) as in Remark \ref{rem:one}, the proof is complete.

\section{Approximation of skew Brownian Motion: convergence of projections}\label{limit}

Here is what we have succeeded in proving so far: in Sections \ref{exten}--\ref{gener} we have established that the invariant spaces of extensions: $\y$ (related to the generator $A_{\al, \be}$) and $\z$ (related to $A_{\al,\be}^\perp$) are complementary. Moreover, projections on $\y$ and $\z$  were proven to be $\p$ and $\q$, respectively.  
Then, in Section \ref{aosb}, we have shown convergence of solution families generated by the operators $A_{n\al, n,\be}$ and $A_{n\al, n,\be}^\perp$, as $n\to \infty$, to the solution families generated by $\askew$ and $\aweks$.

In this, last, section of our paper we provide an epilogue with results concerning convergence of projections: we show that projections $P_{n \al, n\be}$ on the spaces $\mc C_{n\al,n\be}$  related to snapping out Brownian motions, converge strongly to a projection on the space of extensions related to the skew Brownian motion. Here are the details.

\begin{thm}\label{conv_proj1} \ \ 
\begin{itemize}
\item [(a)]
For every $(f_1,f_2)\in \ce$, we have
\begin{align*}
\lim_{n\to \infty} P_{n\al,n\be}(f_1,f_2)=\pp(f_1,f_2),
\end{align*}
where \[\pp(f_1,f_2)\coloneqq \big(f_1^e+\textstyle{\frac{2\al}{\gamma^2}} k_1-\pol k_2, f_2^e+\textstyle{\frac{2\be}{\gamma^2}} k_1+\pol k_2\big)\]
for functions $ k_1,  k_2$ introduced in \eqref{proj:F}.
\item [(b)]
The map $\pp$ is a projection on the space 
\[ \yy \coloneqq \{ (g_1,g_2) \in \ce: g_1^e=g_2^e, \be g_1^o= \al g_2^o \}.\] 
\item [(c)] The space $\yy$ is precisely the subspace of extensions related to the skew Brownian motion  generated by  $\askew$.
\end{itemize}
\end{thm}

\begin{thm} \label{conv_proj2}\ \ 
\begin{itemize}
\item [(a)]
For every $(f_1,f_2)\in \ce$, we have
\begin{align*}
\lim_{n\to \infty} Q_{n\al,n\be}(f_1,f_2)=\qq(f_1,f_2),
\end{align*}
where $\qq\coloneqq I_{\ce} - \pp$.
\item [(b)]
The map $\qq$ is a projection on the space
\[  \zz \coloneqq \{ (g_1,g_2) \in \ce: g_1^e=-g_2^e, \al g_1^o= -\be g_2^o\}.\] 
\item [(c)] The space $\zz$ is the subspace of extensions related to the generator $\aweks$ introduced in Remark \ref{weks:gen}.
\end{itemize}
\end{thm}

\begin{cor}\label{tcztery} The space $\ce $ is a direct sum of two spaces that are invariant under the basic Cartesian product cosine family: 
\begin{align*} \ce = \yy \oplus \zz .\end{align*} \end{cor}
Therefore, limits of complementary transmission conditions \eqref{intro:bc} and \eqref{intro:bc2} are also complementary.

\subsection{Proof of Theorem \ref{conv_proj1} } \ \ 
\textbf{ (a)} From \eqref{intro:eqdef_g1}
 it follows that the first coordinate of $P_{n\al,n\be}(f_1,f_2)$ is
\begin{align*}
g_{1,n}(x) & = \fle^e(x) +\pol(\gamma+2\al)n \int_{-\infty}^x \big[\textstyle{\frac 1\gamma} k_1(y)-\pol  k_2(y) \big]\e^ {-n\sq  (x-y)}   \ud y \nonumber \\
 &\phantom{=} -\pol(\gamma-2\al) n \int_x^{\infty} \big[\textstyle{\frac 1\gamma} k_1(y)+\pol  k_2(y) \big]\e^ {n\sq (x-y)}   \ud y , \qquad x\in \R.  \end{align*}
Lemma \ref{lem:dirnew} in Appendix implies now that $g_{1,n}$ converges in the norm of $\cer$, as $n\to \infty$, to 
$%f_1^e+\pol(\gamma+2\al) \big(\textstyle{\frac 1\gamma}h_1-\pol h_2 \big)-\pol(\gamma-2\al) \big(\textstyle{\frac 1\gamma}h_1+\pol h_2 \big)=
f_1^e+\frac 1\gamma (\frac{2\al}{\gamma} k_1-\frac{\gamma}{2} k_2)$, as desired.

For the proof of convergence of the second coordinate of $P_{n\al,n\be}(f_1,f_2)$, denoted  $g_{2,n}$,  we proceed analogously. Namely, by \eqref{intro:eqdef_g2}, we have
\begin{align*}
g_{2,n}(x) &= \fre^e(x) +\pol (\gamma+2\be)n \int_x^\infty  \big[\textstyle{\frac 1 \gamma} k_1(y)+\pol  k_2(y) \big]\e^{n \sq (x-y)}   \ud y \\
 &\phantom{=} -\pol (\gamma-2\be )n \int_{-\infty}^x \big[\textstyle{\frac 1 \gamma} k_1(y)-\pol  k_2(y) \big]\e^{-n\sq (x-y)}   \ud y , \qquad x\in \R.   
 \end{align*}
Hence, by Lemma \ref{lem:dirnew}, $g_{2,n}$ converges in the norm of $\cer$ to 
$f_2^e+\frac 1\gamma (\frac{2\be}{\gamma}h_1+\frac{\gamma}{2}h_2)$.

\textbf{ (b)} {As a strong limit of projections, $\pp$ is a projection also. To characterize its range we} consider $(f_1,f_2) \in \ce$ and $(g_1,g_2)\coloneqq \pp(f_1,f_2) $.
Since $f_1^e,f_2^e,$ $ k_1$ and $ k_2$ all belong to $\cer$, it follows that so do $g_1$ and $g_2 $. Furthermore,
$g_1^e=f_1^e-\pol  k_2=\pol(f_1^e+f_2^e)=f_2^e+\pol  k_2=g_2^e$.
Similarly, we see that $g_1^o=\frac{2\al}{\gamma^2}  k_1$ and $g_2^o=\frac{2\be}{\gamma^2}  k_1$. 
This shows  that $(\gl,\gr) \in \yy$, that is, that the range of $\pp$ is contained in $\yy$. 

To prove the other inclusion, we assume that  $\be f_1^o= \al f_2^o$ and $f_1^e = f_2^e.$
Then, $2\alpha  k_1 = \gamma^2 f_1^o$, $2\beta  k_1 = \gamma^2 f_2^o$ and $ k_2=0$. Hence the definition of $\gl$ and $\gr$  simplifies to $\gl =f_1^e+f_1^o=f_1$ and $\gr = f_2^e+f_2^o=f_2$. It follows that  the range of $\pp$ contains $\yy$ and the proof of (b) is complete. 

\textbf{ (c)} We observe that the form of the operator $E_{\alpha,\beta}^\textnormal{skew} $ found in Proposition \ref{skew:ext_prop} enables us to characterize the subspace  of extensions related to the skew Brownian motion as the space of  $(\gl,\gr)\in\ce$ satisfying $\gl(x)= \frac{\be-\al}{\apb} \gl(-x)+\frac{2\al}{\apb} \gr(x)$ and $ \gr(-x)=\frac{2\be}{\apb} \gl(-x)+\frac{\al-\be}{\apb} \gr(x) $ for $x\geq 0$.
These relations can be rewritten as
\begin{align*}
\be g_1^o&= \al( g_2-g_1^e),\qquad \al g_2^o= \be( g_2^e-g_1^\es) \quad  \textrm{on } \R^+. 
\end{align*}
Since for any $g\in \mathfrak C(\R)$ we have $g^o+g^e=g$ and $g^e-g^o=g^\es $, subtracting the above equations yields  $g_1^e=g_2^e$ on $\R^+$ and from the first equation we obtain $  \be g_1^o= \al g_2^o$ on $\R^+$.
By \eqref{intro:T}, these relations must in fact hold on $\R$. 
Since the above reasoning can be reversed, it follows that  the space of extensions equals $\yy$. 

\subsection{Proof of Theorem \ref{conv_proj2}}\ \ \ \ \ \ \ 
\textbf{ (a)} This point is an immediate consequence of Theorem \ref{conv_proj1} (a). %and the definition of $Q_{\al,\be}$ given in  \eqref{proj:7}.

\textbf{ (b)} Let $(f_1,f_2) \in \ce$ and $(g_1,g_2)\coloneqq \qq(f_1,f_2) $, that is,
$g_1=f_1^o-\textstyle{\frac{2\al}{\gamma^2}} k_1+\pol k_2$ and $g_2= f_2^o-\textstyle{\frac{2\be}{\gamma^2}} k_1-\pol k_2$.
The fact that $g_1,g_2\in \cer$ follows from point (b) in Theorem \ref{conv_proj1}.
As in  the proof of that result, we obtain
$\gl^e=\pol(f_1^e-f_2^e)=-\gr^e$ and $\al g_1^o=\frac{2\al\be}{\gamma^2}(\be f_1 ^o-\al f_2^o)= -\be g_2^o$ 
implying that $(\gl,\gr) \in \zz.$ 
Moreover, it is easy to check that, if  $(f_1,f_2)\in \zz $, then $\gl=f_1$ and $\gr= f_2$.

\textbf{ (c)}  Since $\zz= \mc J[ \mc C_{\be,\al}^{\textnormal{skew}}]$, the result follows from point (c) in Theorem \ref{conv_proj1} combined with Remark \ref{weks:gen} and the formula $\mc J E_{\be,\al}^\textnormal{skew} J = E_{\al,\be}^\textnormal{weks}$ established in the proof of Theorem \ref{thm:weks}.

\section{Appendix}

The following lemma is used in proving convergence theorems of Sections \ref{aosb}-\ref{limit}.
\begin{lem}\label{lem:dirnew}Let $a>0$, $\phi \in\cer$ and $\varphi \in \cerp$ be given. Then, 
\begin{itemize}
\item [(a)] $\displaystyle \lim_{n \to \infty} \Big [ n a \int_{-\infty}^x \e^{-n a (x-y)}\phi (y) \ud y \Big ]= \phi (x), $ \\
\item [(b)] $ \displaystyle \lim_{n \to \infty} \Big [ n a \int_x^{\infty} \e^{n a (x-y)}\phi (y) \ud y \Big ]= \phi (x) 
$, \item [(c)] \( \displaystyle  \lim_{n \to \infty} \left [ na \int_0^x \e^{-na (x-y)}\varphi (y) \ud y + \e^{-na x}\varphi (0)\right ]= \varphi (x). \)
\end{itemize}
The first two limits are uniform with respect to $x\in \R$, the third is uniform with respect to $x\ge 0.$
\end{lem}
\begin{proof} To prove (a),  we apply a typical argument involving so-called Dirac sequences (see e.g. \cite{lang}*{pp. 227--235} or \cite{feller}*{pp. 219--220}). Namely, given $\eps>0$ we can find a $\delta$ such that $|\phi(x)-\phi(y)|< \frac \eps 2$ as long as $|x-y|< \delta$; this is because members of $\cer$ are uniformly continuous. Having chosen such a $\delta$ we can also find an $n_0$ so large that $\e^{-na \delta} \max (1,2\|\phi\|) < \frac \eps 2$, provided that $n>n_0$, where $\|\phi\|\coloneqq \|\phi\|_{\cer}$.

Next, for any $x \in \R$ we have 
\[ na \int_{-\infty}^x \e^{-na (x-y)} \phi (y) \ud y - \phi (x) =  na \int_0^{\infty} \e^{-na y} [\phi (x-y) -\phi (x)] \ud y .\]
It follows that for $n>n_0$, the absolute value of the left-hand side above does not exceed 
\[ na \int_0^\delta \e^{-na y} {\textstyle \frac \eps 2} \ud y +   na \int_\delta^\infty \e^{-na y} 2\|\phi\| \ud y  = 
{\textstyle \frac \eps 2}(1 - \e^{-na \delta})   + 2 \e^{-na \delta}\| \phi\| < \eps .\]
This completes the proof, $\eps >0$ being arbitrary. 

Condition (b) is  (a) in disguise: (a) becomes (b) if $\phi$ and $x$ are replaced by $\phi^\es$ and $-x$, respectively. Likewise, to obtain (c) we use (a) for $\phi\in \cer$ defined as follows: $\phi (x) = \varphi (x), x \ge 0$ and $\phi (x) = \varphi (0), x <0.$ 
\end{proof}

%\subsection*{Acknowledgment}
%A.B. was supported by the National Science Centre (Poland) grant \linebreak 
%   2017/25/B/ST1/01804, and partly by  National Science Centre grant \linebreak 
%   2016/21/B/ST1/03071.

\bibliographystyle{plain}
%\bibliography{bibliografiakopia}
\bibliography{../../../bibliografia}

\end{document}